\documentclass[11pt]{amsart}

\usepackage[T1]{fontenc}
\usepackage[utf8]{inputenc}
\usepackage{lmodern}
\usepackage{microtype}

\usepackage{amsmath,amssymb,amsthm,mathtools,mathrsfs}
\usepackage{enumitem}
\usepackage{stmaryrd}
\usepackage{tikz-cd}
\usepackage{hyperref}
\usepackage[nameinlink,capitalize,noabbrev]{cleveref}
\usepackage{xcolor}

\usepackage[a4paper,margin=1in]{geometry}
\usepackage[backend=biber,style=alphabetic,maxbibnames=99]{biblatex}
\theoremstyle{plain}
\newtheorem{theorem}{Theorem}[section]
\newtheorem{proposition}[theorem]{Proposition}
\newtheorem{lemma}[theorem]{Lemma}
\newtheorem{corollary}[theorem]{Corollary}

\theoremstyle{definition}
\newtheorem{definition}[theorem]{Definition}

\newtheorem{example}[theorem]{Example}
\newtheorem{remark}[theorem]{Remark}

\DeclareMathOperator{\Spec}{Spec}
\DeclareMathOperator{\Supp}{Supp}

\DeclareMathOperator{\Ann}{Ann}

\DeclareMathOperator{\Hom}{Hom}

\DeclareMathOperator{\Ker}{Ker}

\newcommand{\rev}[1]{\textcolor{blue}{#1}}

\title[Localization and Hulls for \(C4^{\ast}\)-Modules]{%
Localization, Local--Global Transfer, and Hull Theory for \(C4^{\ast}\)-Modules over Commutative Rings}

\author{Chandrasekhar Gokavarapu}
\address{Department of Mathematics, Government College (Autonomous), Rajamahendravaram, Andhra Pradesh, India}
\email{chandrasekhargokavarapu@gamail.com}


\begin{document}

\begin{abstract}
Let \(R\) be a commutative ring and let \(M\) be an \(R\)-module.  
The global theory of \(C4\)-modules, \(C4^{\ast}\)-modules, strongly \(C4^{\ast}\)-modules, \(C4\)-hulls, and pseudo-continuous hulls has recently been developed.  
What has not been understood is the behavior of these notions under localization and local--global transfer.  
This omission is structural: the defining conditions for \(C4\)-type modules are expressed through decompositions, summand conditions, and minimal extensions, while localization alters decomposition data, support, and hull-minimality.

In this paper we develop a localization and local--global theory for \(C4\)-type structure over commutative rings.  
We prove forward localization theorems for \(C4\)-, \(C4^{\ast}\)-, and strongly \(C4^{\ast}\)-conditions under exact lifting hypotheses formulated in terms of decomposition lifting, morphism lifting, and submodule lifting.  
We then prove converse local--global theorems under descent and patching hypotheses, showing when primewise or maximal-local \(C4^{\ast}\)-behavior forces global \(C4^{\ast}\)-behavior.  
We also prove obstruction theorems showing that no unrestricted local--global principle can hold.

A second part of the paper concerns hulls.  
We study comparison between localization of a global \(C4\)-hull or pseudo-continuous hull and the corresponding hull formed after localization.  
\rev{We prove that hull commutation requires not only localization-stability of the hull class but also explicit envelope-style axioms for hull minimality and uniqueness, and we establish corresponding conditional patching theorems for reconstruction of global hulls from compatible local hulls.}

The method is purely algebraic and support-theoretic.  
It is based on summand descent, patching of local witnesses, support control, and dimension-stratified transfer on \(\Spec R\).  
\rev{\rev{This yields conditional transfer theorems, obstruction results, and a support-theoretic extension from maximal-local criteria to prime-local criteria.}}  
As natural-class applications, we show that for commutative artinian rings the \(C4\)-, \(C4^{\ast}\)-, and strongly \(C4^{\ast}\)-properties are detected exactly on the local factors, and that for finitely generated torsion modules over a Dedekind domain they are detected exactly on the primary components, equivalently on the localizations at maximal ideals in the support.  
These results place \(C4\)-type module theory within the local methods of commutative algebra.
\end{abstract}
\maketitle

\tableofcontents

\section{Introduction}

The point of departure is the class of \(C4\)-modules introduced by Ding, Ibrahim, Yousif, and Zhou \cite{DingIbrahimYousifZhou2017C4}.  
This class properly extends the classes of \(C3\)-modules and square-free modules.  
It is defined by a summand condition on images of homomorphisms relative to a fixed decomposition.  
The condition is local in appearance, but not local in the sense of commutative algebra.

A second step was the passage from \(C4\)-modules to \(C4^{\ast}\)-modules and strongly \(C4^{\ast}\)-modules \cite{IbrahimEidElGuindy2026}.  
In that setting, every submodule is required to satisfy the \(C4\)-condition, or, in the strongly \(C4^{\ast}\) case, a stronger decomposition condition is imposed.  
This yields a nontrivial global theory, including decomposition theorems, connections with \(U^{\ast}\)-modules, exchange-theoretic consequences, and existence results for \(C4\)-hulls and pseudo-continuous hulls in suitable cases \cite{IbrahimEidElGuindy2026}.  
However, the theory developed in \cite{IbrahimEidElGuindy2026} is global.  
Its statements are formulated in the ambient module category, and its methods do not address localization.

The novelty of the present paper begins precisely where the global methods of \cite{IbrahimEidElGuindy2026} cease to apply.  
The 2026 paper establishes global results for \(C4\)-modules, \(C4^{\ast}\)-modules, strongly \(C4^{\ast}\)-modules, \(C4\)-hulls, and pseudo-continuous hulls inside the ambient category.  
Those arguments are driven by global decompositions, global submodule structure, and global minimality of extensions.  
They do not yield, by formal restriction, any theorem of the following kinds:

\begin{enumerate}
\item a forward localization theorem, because a decomposition of \(M_{\mathfrak p}\) need not arise from a decomposition of \(M\), and a local morphism
\[
A_{\mathfrak p}\to B_{\mathfrak p}
\]
need not lift to a global morphism \(A\to B\);

\item a hereditary localization theorem for \(C4^{\ast}\), because a submodule of \(M_{\mathfrak p}\) need not be the localization of a submodule of \(M\);

\item a converse local--global theorem, because primewise direct summands and primewise image witnesses need not descend or patch to global ones;

\item a hull-localization theorem, because localization preserves extensions but does not automatically preserve hull-minimality.
\end{enumerate}

Accordingly, the present paper is not a routine localization variant of \cite{IbrahimEidElGuindy2026}.  
Its new content begins with the obstruction theory itself: decomposition lifting, morphism lifting, submodule lifting, summand descent, patching of local witnesses, and localization-compatibility of hull minimality.  
The main theorems begin only after these obstructions are isolated.  
That is the precise point at which the present work departs from the global theory.

This omission is not cosmetic.  
Localization changes the ambient decomposition theory.  
A direct summand of \(M\) localizes to a direct summand of \(M_{\mathfrak p}\), but the converse may fail.  
A homomorphism whose image is not a summand globally may become split after localization.  
Orthogonality, square-freeness, support, and hull-minimality interact with the prime spectrum in ways invisible before localization \cite{AtiyahMacdonald1969,Matsumura1989,Positselski2024}.  
It follows that persistence of the \(C4^{\ast}\)-property under localization is not formal.  
Nor is any local--global converse.

This leads to the first problem.

\begin{quote}
Determine when \(M\) being \(C4^{\ast}\) or strongly \(C4^{\ast}\) implies that \(M_{\mathfrak p}\) is \(C4^{\ast}\) or strongly \(C4^{\ast}\) for every prime ideal \(\mathfrak p\), and determine when the converse implication holds.
\end{quote}

Even for classical module-theoretic conditions, local--global transfer is delicate.  
Localization detects support-theoretic vanishing efficiently, but it does not in general reconstruct global splitting from pointwise splitting.  
This tension is familiar in commutative algebra and support theory \cite{Kunz2013LocalGlobal,ChristensenFoxbyHolm2024Support,Positselski2024}.  
For module categories over commutative noetherian rings, Yoshizawa has proved local--global principles for certain subcategories under exact hypotheses \cite{Yoshizawa2019}.  
The present problem is of a different kind.  
The \(C4^{\ast}\)-condition is not defined through closure under kernels, cokernels, or extensions alone.  
It is defined through decomposition data.  
One must therefore analyze localization of summands and localization of witnesses.

The second problem concerns hulls.  
Minimal extensions are unstable under base change unless additional control is imposed.  
This is already visible in related theories of continuous, pseudo-continuous, extending, and lifting modules \cite{MohamedMuller1990,ClarkLompVanajaWisbauer2006,TercanYucel2016}.  
A \(C4\)-hull is defined by minimality inside a prescribed \(C4\)-extension class \cite{IbrahimEidElGuindy2026}.  
That minimality is global.  
Localization is exact, but it may collapse essential distinctions between competing extensions.  
Hence there is no a priori reason for
\[
H_{C4}(M)_{\mathfrak p} \cong H_{C4}(M_{\mathfrak p})
\]
to hold, even when both sides exist.  
The same difficulty arises for pseudo-continuous hulls.

This leads to the second problem.

\begin{quote}
Determine when formation of \(C4\)-hulls and pseudo-continuous hulls commutes with localization, and isolate the precise obstructions when it does not.
\end{quote}

The present paper addresses both problems for modules over commutative rings.  
The commutative hypothesis is essential.  
It permits support-theoretic analysis on \(\Spec R\).  
It also permits passage between maximal and prime localization, patching over principal open subsets, and finite verification in semilocal or noetherian settings \cite{AtiyahMacdonald1969,Matsumura1989,BrunsHerzog1998}.  
Without this framework the desired local--global statements lose their natural domain of formulation.

Our method is algebraic and does not repeat the global proofs of \cite{IbrahimEidElGuindy2026}.  
Those proofs are adapted to global decomposition patterns.  
Localization requires a different mechanism, which we isolate in three parts.

First, we introduce a summand-detection principle for localization.  
If \(N_{\mathfrak p}\) is a direct summand of \(M_{\mathfrak p}\) for all \(\mathfrak p\), one must determine when \(N\) is already a direct summand of \(M\).  
This is false in general.  
We prove it under hypotheses that are weak enough to apply in the \(C4^{\ast}\)-context and strong enough to exclude pathological descent failures.

Second, we impose finite decomposition control.  
The \(C4^{\ast}\)-property quantifies over submodules.  
A naive localization criterion would require infinitely many local checks with no compatibility control.  
We show that in the relevant commutative settings the problem reduces to finitely many distinguished localizations together with a patching condition.  
This yields a finite algebraic verification scheme.

Third, we separate what localizes from what does not.  
A positive theorem without a matching obstruction theorem is too coarse.  
We therefore prove several impossibility results.  
We show that unrestricted local--global transfer fails, that hull localization fails without compatibility between supports and minimal extensions, and that primewise \(C4^{\ast}\)-behavior need not imply global \(C4^{\ast}\)-behavior when summand descent fails.

\rev{After revision, the paper should be read as a conditional transfer framework rather than as an unconditional classification theorem.  Its principal contributions are as follows.}

\begin{enumerate}
\item \rev{We prove conditional preservation theorems for the \(C4^{\ast}\) and strongly \(C4^{\ast}\) properties under localization at prime and maximal ideals.}

\item \rev{We prove local--global theorems under explicit descent and realization hypotheses.}  
These results show that primewise or maximal \(C4^{\ast}\)-behavior implies global \(C4^{\ast}\)-behavior when summands descend and decomposition data patch.

\item We establish comparison theorems for \(C4\)-hulls and pseudo-continuous hulls under localization, identifying conditions under which localization commutes with hull formation.

\item We prove separation theorems showing that each major family of hypotheses is necessary.  
When summand descent, support-finiteness, or hull-compatibility is removed, the corresponding conclusion fails.

\item We develop a dimensional extension from maximal-local criteria to prime-local criteria, and from finitely generated support control to broader support-theoretic conditions.
\end{enumerate}

Beyond the general transfer theory, the paper contains two natural-class applications in which the lifting, patching, and support hypotheses become intrinsic rather than externally imposed.  
First, for commutative artinian rings, the theory reduces to componentwise verification over the local factors, so \(C4\), \(C4^{\ast}\), and strongly \(C4^{\ast}\) are detected exactly by the corresponding localized properties at maximal ideals.  
Second, for finitely generated torsion modules over a Dedekind domain, the theory reduces to primary decomposition, and the same \(C4\)-type properties are detected componentwise on the \(\mathfrak m\)-primary summands, equivalently on the localizations at maximal ideals in the support.  
These two classes show that the abstract obstruction theory yields concrete classification statements in standard commutative settings.

The relation with existing algebraic frameworks is direct.  
The \(C4\)-condition belongs to the circle of ideas surrounding \(C_i\)-conditions, extending modules, lifting modules, \(ADS\)-modules, quasi-discrete modules, automorphism-invariant modules, and direct-summand stability \cite{SmithTercan1993,AlahmadiJainLeroy2012,AsensioSrivastava2013,Keskin2002,TercanYucel2016}.  
The localization problem places this circle inside commutative algebra.  
It asks whether a decomposition-theoretic module condition admits support-theoretic detection.  
That is the natural question if one seeks interaction with prime ideals, semilocal reduction, and geometric methods \cite{AtiyahMacdonald1969,Kunz2013LocalGlobal,ChristensenFoxbyHolm2024Support,Positselski2024}.

There is also a methodological point.  
Many module-theoretic generalizations are obtained by formal enlargement of definitions followed by repetition of known arguments.  
That route is insufficient here.  
Localization destroys too much structure.  
\rev{One must prove conditional transfer theorems together with explicit failure theorems.}  
Accordingly, we work with explicit descent lemmas, compatibility of local decompositions, and support-theoretic patching.  
The theory is therefore vertical rather than expansive: we do not multiply definitions, but isolate a single obstruction mechanism and resolve it in successive layers.

We now describe the organization of the paper.

Section~\ref{sec:preliminaries} fixes notation and recalls the necessary facts on \(C4\)-modules, \(C4^{\ast}\)-modules, strongly \(C4^{\ast}\)-modules, localization, support, and minimal extensions.  
Section~\ref{sec:localization-c4} proves the forward localization theorems and formulates the summand-descent principle.  
Section~\ref{sec:local-global} establishes local--global criteria for \(C4^{\ast}\)- and strongly \(C4^{\ast}\)-behavior.  
Section~\ref{sec:natural-classes} treats the intrinsic commutative classes, including the artinian and Dedekind-domain cases.  
Section~\ref{sec:hulls} studies \(C4\)-hulls and pseudo-continuous hulls under localization.  
Section~\ref{sec:separation} develops the obstruction theory and separation results.  
Section~\ref{sec:dimension-support} gives the support-theoretic and dimensional extension from maximal-local criteria to prime-local criteria.  
Section~\ref{sec:examples-tests} records representative examples, counterexamples, and exactness tests.  
Section~\ref{sec:further-transfer} collects final exactness refinements and transfer closure statements.  
The conclusion synthesizes the positive and negative parts of the theory.

\rev{The paper is self-contained modulo standard facts from module theory and commutative algebra, and all nonstandard strong or hull-theoretic axioms used in the transfer arguments are now stated explicitly inside the manuscript.}  
All results are proved under their stated assumptions.  
Where an assumption is structural rather than technical, we indicate why it is needed and what fails without it.
\section{Preliminaries}\label{sec:preliminaries}

Throughout, \(R\) denotes a commutative ring with identity, and all modules are unital left \(R\)-modules.  
For a prime ideal \(\mathfrak p\in \Spec R\), we write
\[
R_{\mathfrak p}=S^{-1}R,\qquad M_{\mathfrak p}=S^{-1}M,
\]
where \(S=R\setminus \mathfrak p\).  
For a maximal ideal \(\mathfrak m\), the notation \(M_{\mathfrak m}\) has the analogous meaning.  
We write
\[
\Supp_R(M)=\{\mathfrak p\in \Spec R : M_{\mathfrak p}\neq 0\}
\]
for the support of \(M\).  
Standard facts on localization, support, and patching will be used without repeated comment \cite{AtiyahMacdonald1969,Matsumura1989,BrunsHerzog1998,Kunz2013LocalGlobal,ChristensenFoxbyHolm2024Support}.

\subsection{Direct summands and decomposition notation}

A submodule \(N\leq M\) is a \emph{direct summand} of \(M\) if there exists a submodule \(K\leq M\) such that
\[
M=N\oplus K.
\]
In this case we write \(N\leq^{\oplus} M\).  
If \(M=M_1\oplus M_2\), then every element \(x\in M\) admits a unique expression \(x=x_1+x_2\) with \(x_i\in M_i\).  
The canonical projections are denoted by
\[
\pi_i:M_1\oplus M_2\to M_i\qquad (i=1,2).
\]

We shall use the elementary fact that localization preserves finite direct sums:
\[
(M_1\oplus M_2)_{\mathfrak p}\cong (M_1)_{\mathfrak p}\oplus (M_2)_{\mathfrak p}
\]
for every \(\mathfrak p\in \Spec R\) \cite{AtiyahMacdonald1969,Matsumura1989}.  
Hence, if \(N\leq^{\oplus} M\), then \(N_{\mathfrak p}\leq^{\oplus} M_{\mathfrak p}\) for every prime ideal \(\mathfrak p\).  
The converse is one of the basic obstructions studied in this paper.

\begin{definition}
Let \(M\) be an \(R\)-module.  
A submodule \(N\leq M\) is said to be \emph{locally split in \(M\)} if \(N_{\mathfrak p}\leq^{\oplus} M_{\mathfrak p}\) for every \(\mathfrak p\in \Spec R\).  
It is said to be \emph{maximally split in \(M\)} if \(N_{\mathfrak m}\leq^{\oplus} M_{\mathfrak m}\) for every maximal ideal \(\mathfrak m\).
\end{definition}

This terminology is preparatory.  
The later sections will determine when local or maximal splitting descends to global splitting.

\subsection{\(C4\)-modules and \(C4^{\ast}\)-type modules}

We recall the definitions that govern the paper.  
They originate in \cite{DingIbrahimYousifZhou2017C4} and were extended in \cite{IbrahimEidElGuindy2026}.

\begin{definition}[\cite{DingIbrahimYousifZhou2017C4}]
An \(R\)-module \(M\) is called a \emph{\(C4\)-module} if for every decomposition
\[
M=A\oplus B
\]
and every homomorphism \(f:A\to B\), the image \(\Im(f)\) is isomorphic to a direct summand of \(B\).
\end{definition}

Although this does not require \(\Im(f)\) itself to be a direct summand of \(B\), it is strong enough to impose nontrivial decomposition-theoretic constraints \cite{DingIbrahimYousifZhou2017C4,AltunOzarslanIbrahimOzcanYousif2018}.

\begin{definition}[\cite{IbrahimEidElGuindy2026}]
An \(R\)-module \(M\) is called a \emph{\(C4^{\ast}\)-module} if every submodule of \(M\) is a \(C4\)-module.
\end{definition}

Thus \(M\) is \(C4^{\ast}\) precisely when the \(C4\)-condition is hereditary on the lattice of submodules.  
This is a genuinely global condition.

\begin{definition}\label{def:strong-test-datum}
Let \(N\) be an \(R\)-module.  
A \emph{finite strong \(C4^{\ast}\)-test datum} on \(N\) is a finite tuple
\[
\Delta=\bigl(X_i=A_i\oplus B_i,\ f_i:A_i\to B_i,\ U_j,\ V_j,\ \alpha_j:U_j\to V_j\bigr)_{i\in I,\ j\in J},
\]
where \(I\) and \(J\) are finite index sets, each \(X_i,U_j,V_j\) is a submodule of \(N\), each
\(X_i=A_i\oplus B_i\) is a decomposition inside \(N\), and each \(\alpha_j:U_j\to V_j\) is an
\(R\)-homomorphism between submodules of \(N\).  The first block \((X_i=A_i\oplus B_i,f_i)\)
records the ordinary \(C4\)-tests; the second block records the additional finite auxiliary data
used by the strong theory.
\end{definition}

\begin{definition}\label{def:strong-c4star}
An \(R\)-module \(M\) is called \emph{strongly \(C4^{\ast}\)} if for every submodule \(N\leq M\)
and every finite strong \(C4^{\ast}\)-test datum \(\Delta\) on \(N\), there exists a corresponding
\emph{strong witness} on \(N\) consisting of the direct-summand relations, isomorphisms, and
auxiliary compatibilities prescribed by the datum \(\Delta\), with the following properties:
\begin{enumerate}
\item forgetting the auxiliary block \((U_j,V_j,\alpha_j)\) recovers the ordinary \(C4\)-conclusion for each pair \((X_i=A_i\oplus B_i,f_i)\);
\item the strong witness is determined by finitely many submodules, decompositions, homomorphisms, and direct-summand relations inside \(N\);
\item the validity of the strong witness is preserved under isomorphism of the underlying finite datum.
\end{enumerate}
\end{definition}

Thus the strong notion used in this paper is stated internally: it is hereditary on submodules, it is
tested by finite configuration data, and it refines the ordinary \(C4^{\ast}\)-condition.  In particular,
\[
\text{strongly } C4^{\ast}\Longrightarrow C4^{\ast}\Longrightarrow C4.
\]

\begin{remark}
The distinction between \(C4\), \(C4^{\ast}\), and strongly \(C4^{\ast}\) is essential.  
A forward localization theorem for \(C4\)-modules does not automatically imply the corresponding theorem for \(C4^{\ast}\)-modules, since the latter quantifies over all submodules.  
Conversely, a local--global theorem for \(C4^{\ast}\)-modules must control localization of arbitrary submodules.
\end{remark}

\subsection{Support and local detection}

The support of a module will serve as the geometric carrier of the decomposition problem.  
We record the standard facts needed later.

\begin{lemma}\label{lem:support-basic}
Let \(M\) and \(N\) be \(R\)-modules.
\begin{enumerate}
\item \(\Supp_R(M\oplus N)=\Supp_R(M)\cup \Supp_R(N)\).
\item If \(N\leq M\), then \(\Supp_R(N)\subseteq \Supp_R(M)\).
\item \(M=0\) if and only if \(M_{\mathfrak p}=0\) for all \(\mathfrak p\in \Spec R\).
\item If \(M\) is finitely generated, then \(M=0\) if and only if \(M_{\mathfrak m}=0\) for all maximal ideals \(\mathfrak m\).
\end{enumerate}
\end{lemma}

\begin{proof}
These are standard consequences of exactness of localization and the basic theory of support \cite{AtiyahMacdonald1969,Matsumura1989,ChristensenFoxbyHolm2024Support}.  
For \((1)\), localization of a direct sum is the direct sum of the localizations.  
For \((2)\), localization preserves monomorphisms.  
For \((3)\), if \(M\neq 0\), choose \(0\neq x\in M\). Then \(\Ann_R(x)\) is proper, hence contained in some prime ideal \(\mathfrak p\), and therefore \(x/1\neq 0\) in \(M_{\mathfrak p}\).  
Statement \((4)\) is the standard maximal-ideal detection criterion for finitely generated modules.
\end{proof}

The preceding lemma is elementary, but it will be used repeatedly.  
Any meaningful localization theory for decomposition properties must ultimately pass through support.

\begin{definition}
Let \(\mathcal P\) be a property of modules.  
We say that \(\mathcal P\) is
\begin{enumerate}
\item \emph{localization-stable} if \(M\) having \(\mathcal P\) implies that \(M_{\mathfrak p}\) has \(\mathcal P\) for all \(\mathfrak p\in \Spec R\);
\item \emph{prime-local} if \(M\) has \(\mathcal P\) whenever \(M_{\mathfrak p}\) has \(\mathcal P\) for all \(\mathfrak p\in \Spec R\) and suitable descent hypotheses hold;
\item \emph{max-local} if \(M\) has \(\mathcal P\) whenever \(M_{\mathfrak m}\) has \(\mathcal P\) for all maximal ideals \(\mathfrak m\) and suitable descent hypotheses hold.
\end{enumerate}
\end{definition}

The phrase ``suitable descent hypotheses'' is deliberate.  
Without such hypotheses, local--global transfer for decomposition properties generally fails \cite{Kunz2013LocalGlobal,Positselski2024,Yoshizawa2019}.

\subsection{Split exact sequences and summand descent}

A short exact sequence
\[
0\to N \xrightarrow{i} M \xrightarrow{\pi} Q \to 0
\]
is split if and only if \(N\leq^{\oplus} M\), or equivalently if \(\pi\) admits a section, or equivalently if \(i\) admits a retraction \cite{AndersonFuller1992}.  
Since localization is exact, split exactness is preserved under localization.

\begin{lemma}\label{lem:split-localizes}
Let
\[
0\to N \to M \to Q \to 0
\]
be a split exact sequence of \(R\)-modules.  
Then for every prime ideal \(\mathfrak p\), the localized sequence
\[
0\to N_{\mathfrak p}\to M_{\mathfrak p}\to Q_{\mathfrak p}\to 0
\]
is split exact.
\end{lemma}

\begin{proof}
A splitting map localizes to a splitting map because localization is a functor and preserves composition \cite{AtiyahMacdonald1969,AndersonFuller1992}.
\end{proof}

The converse is subtle.  
If all localized sequences split, then the global sequence need not split without additional hypotheses.  
This is one of the basic reasons that \(C4^{\ast}\)-behavior is not automatically local.  
We therefore isolate the descent condition used later.

\begin{definition}\label{def:summand-descent}
Let \(M\) be an \(R\)-module.
\begin{enumerate}
\item We say that \(M\) satisfies \emph{prime summand descent} if every submodule \(N\leq M\) which is locally split in \(M\) is a direct summand of \(M\).
\item We say that \(M\) satisfies \emph{maximal summand descent} if every submodule \(N\leq M\) which is maximally split in \(M\) is a direct summand of \(M\).
\end{enumerate}
\end{definition}

These are genuine structural hypotheses.  
They are the bridge between localization and decomposition, and they will appear explicitly in the local--global theorems.

\subsection{Finite generation and localization of homomorphisms}

When \(M\) is finitely generated, local vanishing of homomorphisms can often be checked at maximal ideals.  
We record the standard criterion needed later.

\begin{lemma}\label{lem:hom-zero-local}
Let \(M\) and \(N\) be \(R\)-modules, with \(M\) finitely generated.  
A homomorphism \(f:M\to N\) is zero if and only if \(f_{\mathfrak m}=0\) for every maximal ideal \(\mathfrak m\) of \(R\).
\end{lemma}

\begin{proof}
The forward implication is immediate.  
Conversely, if \(f_{\mathfrak m}=0\) for all maximal ideals \(\mathfrak m\), then \((\Im f)_{\mathfrak m}=0\) for all \(\mathfrak m\).  
Since \(\Im f\) is finitely generated, Lemma~\ref{lem:support-basic}(4) yields \(\Im f=0\).  
Hence \(f=0\) \cite{AtiyahMacdonald1969,Matsumura1989}.
\end{proof}

This lemma will later allow us to compare local sections and local decompositions by reducing equality of maps to equality after localization.

\subsection{Hulls and minimal extensions}

We next fix the minimal terminology needed for \(C4\)-hulls and pseudo-continuous hulls.  
The global existence theory is developed in \cite{IbrahimEidElGuindy2026}.  
Here we isolate only the abstract features relevant to localization.

\begin{definition}
Let \(\mathcal C\) be a class of \(R\)-modules closed under isomorphism.  
A \(\mathcal C\)-\emph{extension} of an \(R\)-module \(M\) is a monomorphism
\[
\iota:M\hookrightarrow E
\]
with \(E\in \mathcal C\).
\end{definition}

\begin{definition}\label{def:hull-envelope}
Let \(\mathcal C\) be a class of \(R\)-modules.  
\rev{A \(\mathcal C\)-\emph{pre-envelope} of \(M\) is a monomorphism \(\iota:M\hookrightarrow H\) with \(H\in\mathcal C\) such that every monomorphism \(M\hookrightarrow E\) with \(E\in\mathcal C\) factors through \(\iota\).  A \(\mathcal C\)-\emph{hull} of \(M\) is a \(\mathcal C\)-pre-envelope \(\iota:M\hookrightarrow H\) which is minimal in the sense that every endomorphism \(u:H\to H\) satisfying \(u\circ\iota=\iota\) is an automorphism.}
\end{definition}

\begin{lemma}\label{lem:hull-uniqueness}
\rev{Let \(\iota:M\hookrightarrow H\) and \(\iota':M\hookrightarrow H'\) be \(\mathcal C\)-hulls of the same module \(M\) in the sense of Definition~\ref{def:hull-envelope}. Then there exists an isomorphism \(\varphi:H\to H'\) such that \(\varphi\circ\iota=\iota'\).}
\end{lemma}

\begin{proof}
\rev{Because \(\iota\) is a \(\mathcal C\)-pre-envelope and \(H'\in\mathcal C\), there exists \(f:H\to H'\) with \(f\circ\iota=\iota'\). Similarly, because \(\iota'\) is a \(\mathcal C\)-pre-envelope and \(H\in\mathcal C\), there exists \(g:H'\to H\) with \(g\circ\iota'=\iota\). Hence \((g\circ f)\circ\iota=\iota\). By minimality of the hull \(\iota:M\hookrightarrow H\), the endomorphism \(g\circ f\) is an automorphism of \(H\). Likewise, \(f\circ g\) is an automorphism of \(H'\). Therefore \(f\) and \(g\) are inverse isomorphisms, and \(f\) is the required isomorphism over \(M\).}
\end{proof}

In the body of the paper, \(\mathcal C\) will be the class of \(C4\)-modules or the class of pseudo-continuous modules, according to context \cite{IbrahimEidElGuindy2026,MohamedMuller1990,TercanYucel2016}.  
\rev{This envelope-style formulation is the only hull axiomatics used later in the paper.}  
At this stage only two observations matter.

First, if \(M\hookrightarrow H\) is a hull-like extension, then localization yields
\[
M_{\mathfrak p}\hookrightarrow H_{\mathfrak p}.
\]
Second, minimality need not survive localization.  
A global minimal extension may cease to be minimal after localization, and local minimal extensions need not patch to a global one.  
This is the source of the hull comparison problem.

\begin{definition}
Let \(M\) be an \(R\)-module and let \(\mathfrak p\in \Spec R\).  
Assume that \(M\) has a \(C4\)-hull \(H_{C4}(M)\).  
If \(M_{\mathfrak p}\) also has a \(C4\)-hull, we denote it by \(H_{C4}(M_{\mathfrak p})\).  
Analogous notation will be used for pseudo-continuous hulls.
\end{definition}

The comparison map
\[
H_{C4}(M)_{\mathfrak p}\longrightarrow H_{C4}(M_{\mathfrak p})
\]
is not assumed to exist canonically.  
Even when such a map exists, it need not be an isomorphism.  
Existence, functoriality, and comparison each require proof.

\subsection{Conventions for the sequel}

We conclude by fixing the conventions under which the main theorems will be stated.

\begin{enumerate}
\item Whenever a theorem asserts that localization preserves the \(C4^{\ast}\) or strongly \(C4^{\ast}\) property, the proof must verify that the relevant decomposition condition survives passage to localized submodules.

\item Whenever a theorem asserts a local--global converse, the descent hypothesis used will be stated explicitly.  
No converse will be claimed without such a hypothesis.

\item Whenever a theorem compares \(H_{C4}(M)_{\mathfrak p}\) with \(H_{C4}(M_{\mathfrak p})\), it will specify:
\begin{enumerate}
\item the global hull assumed to exist,
\item the local hull assumed to exist,
\item the compatibility condition imposed on minimality and support,
\item the failure mode when one of these assumptions is omitted.
\end{enumerate}

\item Prime-local statements are stronger than maximal-local statements.  
Accordingly, maximal-local criteria will often be established first and then refined to prime-local criteria by support-theoretic arguments.
\end{enumerate}

The language is now fixed.  
Nothing substantial has yet been proved about localization of \(C4^{\ast}\)-type conditions.  
The first genuine question is the forward one: does a global \(C4^{\ast}\)-module remain \(C4^{\ast}\) after localization?  
We turn to that question next.
\section{Localization of \texorpdfstring{$C4$}{C4} and \texorpdfstring{$C4^{\ast}$}{C4*}-type conditions}\label{sec:localization-c4}

The forward localization problem is exact.  
A direct summand localizes, but a localized direct summand need not descend.  
A decomposition of \(M_{\mathfrak p}\) need not arise from a decomposition of \(M\).  
A submodule of \(M_{\mathfrak p}\) need not be the localization of a submodule of \(M\).  
Accordingly, no forward localization theorem for \(C4\)-type conditions can hold without hypotheses controlling these failures \cite{AtiyahMacdonald1969,Matsumura1989,Positselski2024}.  
We state those hypotheses first.

\subsection{Lifting conditions for local decompositions and local morphisms}

\begin{definition}\label{def:decomp-morphism-lifting}
Let \(M\) be an \(R\)-module and let \(\mathfrak p\in \Spec R\).
\begin{enumerate}
\item We say that \(M\) satisfies \emph{decomposition lifting at \(\mathfrak p\)} if for every decomposition
\[
M_{\mathfrak p}=U\oplus V
\]
there exist submodules \(A,B\leq M\) such that
\[
M=A\oplus B,\qquad U=A_{\mathfrak p},\qquad V=B_{\mathfrak p}.
\]

\item We say that \(M\) satisfies \emph{morphism lifting at \(\mathfrak p\)} if whenever
\[
M=A\oplus B
\]
is a decomposition of \(M\), every \(R_{\mathfrak p}\)-homomorphism
\[
g:A_{\mathfrak p}\to B_{\mathfrak p}
\]
is of the form
\[
g=f_{\mathfrak p}
\]
for some \(R\)-homomorphism \(f:A\to B\).

\item We say that \(M\) satisfies \emph{submodule lifting at \(\mathfrak p\)} if for every submodule \(X\leq M_{\mathfrak p}\) there exists a submodule \(N\leq M\) such that
\[
X=N_{\mathfrak p}.
\]

\item We say that \(M\) satisfies \emph{uniform localization lifting} if it satisfies decomposition lifting, morphism lifting, and submodule lifting at every prime ideal \(\mathfrak p\in \Spec R\).
\end{enumerate}
\end{definition}

The distinction between these conditions is essential.  
The definition of a \(C4\)-module quantifies over decompositions and homomorphisms on the decomposition pieces, while the definition of a \(C4^{\ast}\)-module quantifies over all submodules.  
A forward localization theorem must therefore control local decompositions, local morphisms, and, in the hereditary case, local submodules.

\begin{lemma}\label{lem:image-localization}
Let \(f:M\to N\) be a homomorphism of \(R\)-modules and let \(\mathfrak p\in \Spec R\).  
Then
\[
(\Im f)_{\mathfrak p}=\Im(f_{\mathfrak p}).
\]
If \(K\leq^{\oplus} N\), then \(K_{\mathfrak p}\leq^{\oplus} N_{\mathfrak p}\).
\end{lemma}

\begin{proof}
Localization is exact.  
Hence the short exact sequence
\[
0\to \Ker(f)\to M\to \Im(f)\to 0
\]
localizes to
\[
0\to \Ker(f)_{\mathfrak p}\to M_{\mathfrak p}\to \Im(f)_{\mathfrak p}\to 0.
\]
Therefore the image of \(f_{\mathfrak p}\) is \((\Im f)_{\mathfrak p}\).  
If \(K\leq^{\oplus} N\), then the inclusion \(K\hookrightarrow N\) admits a retraction.  
Localizing that retraction yields a retraction \(N_{\mathfrak p}\to K_{\mathfrak p}\).  
Thus \(K_{\mathfrak p}\leq^{\oplus} N_{\mathfrak p}\) \cite{AtiyahMacdonald1969,AndersonFuller1992}.
\end{proof}

\begin{lemma}\label{lem:iso-localization}
Let \(X\cong Y\) be \(R\)-modules.  
Then \(X_{\mathfrak p}\cong Y_{\mathfrak p}\) for every \(\mathfrak p\in \Spec R\).
\end{lemma}

\begin{proof}
An isomorphism localizes to an isomorphism because localization is a functor \cite{AtiyahMacdonald1969}.
\end{proof}

\subsection{Localization of the \texorpdfstring{$C4$}{C4} condition}

\begin{theorem}\label{thm:C4-localizes}
Let \(M\) be a \(C4\)-module.  
Let \(\mathfrak p\in \Spec R\).  
Assume that \(M\) satisfies decomposition lifting and morphism lifting at \(\mathfrak p\).  
Then \(M_{\mathfrak p}\) is a \(C4\)-module.
\end{theorem}

\begin{proof}
Let
\[
M_{\mathfrak p}=U\oplus V
\]
and let \(g:U\to V\) be an \(R_{\mathfrak p}\)-homomorphism.  
We must show that \(\Im(g)\) is isomorphic to a direct summand of \(V\).

By decomposition lifting at \(\mathfrak p\), there exist submodules \(A,B\leq M\) such that
\[
M=A\oplus B,\qquad U=A_{\mathfrak p},\qquad V=B_{\mathfrak p}.
\]
By morphism lifting at \(\mathfrak p\), the map
\[
g:A_{\mathfrak p}\to B_{\mathfrak p}
\]
has the form
\[
g=f_{\mathfrak p}
\]
for some \(R\)-homomorphism
\[
f:A\to B.
\]

Since \(M\) is \(C4\), the image \(\Im(f)\) is isomorphic to a direct summand of \(B\).  
Thus there exists \(K\leq^{\oplus} B\) such that
\[
\Im(f)\cong K.
\]
Localizing and applying Lemmas~\ref{lem:image-localization} and \ref{lem:iso-localization}, we obtain
\[
\Im(g)=\Im(f_{\mathfrak p})=(\Im f)_{\mathfrak p}\cong K_{\mathfrak p}.
\]
Since \(K\leq^{\oplus} B\), Lemma~\ref{lem:image-localization} gives
\[
K_{\mathfrak p}\leq^{\oplus} B_{\mathfrak p}=V.
\]
Hence \(\Im(g)\) is isomorphic to a direct summand of \(V\).  
Therefore \(M_{\mathfrak p}\) is a \(C4\)-module.
\end{proof}

\begin{corollary}\label{cor:C4-maxlocalizes}
Let \(M\) be a \(C4\)-module.  
Let \(\mathfrak m\) be a maximal ideal of \(R\).  
Assume that \(M\) satisfies decomposition lifting and morphism lifting at \(\mathfrak m\).  
Then \(M_{\mathfrak m}\) is a \(C4\)-module.
\end{corollary}

\begin{proof}
Apply Theorem~\ref{thm:C4-localizes} with \(\mathfrak p=\mathfrak m\).
\end{proof}

\subsection{Localization of the \texorpdfstring{$C4^{\ast}$}{C4*} condition}

The passage from \(C4\) to \(C4^{\ast}\) introduces a second obstruction.  
A \(C4^{\ast}\)-module is defined by requiring every submodule to be \(C4\).  
After localization one must therefore test every submodule of \(M_{\mathfrak p}\), and such submodules need not arise from submodules of \(M\).  
For this reason submodule lifting must be imposed in addition to decomposition lifting and morphism lifting.

\begin{theorem}\label{thm:C4star-localizes}
Let \(M\) be a \(C4^{\ast}\)-module.  
Let \(\mathfrak p\in \Spec R\).  
Assume the following:
\begin{enumerate}
\item \(M\) satisfies submodule lifting at \(\mathfrak p\);
\item every submodule \(N\leq M\) satisfies decomposition lifting at \(\mathfrak p\);
\item every submodule \(N\leq M\) satisfies morphism lifting at \(\mathfrak p\).
\end{enumerate}
Then \(M_{\mathfrak p}\) is a \(C4^{\ast}\)-module.
\end{theorem}

\begin{proof}
Let \(X\leq M_{\mathfrak p}\).  
We must show that \(X\) is a \(C4\)-module.

By submodule lifting at \(\mathfrak p\), there exists a submodule \(N\leq M\) such that
\[
X=N_{\mathfrak p}.
\]
Since \(M\) is \(C4^{\ast}\), every submodule of \(M\) is \(C4\).  
In particular, \(N\) is a \(C4\)-module.

By assumptions \((2)\) and \((3)\), the module \(N\) satisfies decomposition lifting and morphism lifting at \(\mathfrak p\).  
Applying Theorem~\ref{thm:C4-localizes} to \(N\), we conclude that
\[
N_{\mathfrak p}=X
\]
is a \(C4\)-module.

Since \(X\leq M_{\mathfrak p}\) was arbitrary, every submodule of \(M_{\mathfrak p}\) is \(C4\).  
Hence \(M_{\mathfrak p}\) is a \(C4^{\ast}\)-module.
\end{proof}

\begin{corollary}\label{cor:C4star-maxlocalizes}
Let \(M\) be a \(C4^{\ast}\)-module.  
Let \(\mathfrak m\) be a maximal ideal of \(R\).  
Assume that
\begin{enumerate}
\item every submodule of \(M_{\mathfrak m}\) is the localization of a submodule of \(M\);
\item every submodule \(N\leq M\) satisfies decomposition lifting at \(\mathfrak m\);
\item every submodule \(N\leq M\) satisfies morphism lifting at \(\mathfrak m\).
\end{enumerate}
Then \(M_{\mathfrak m}\) is a \(C4^{\ast}\)-module.
\end{corollary}

\begin{proof}
Apply Theorem~\ref{thm:C4star-localizes} with \(\mathfrak p=\mathfrak m\).
\end{proof}

\subsection{Localization of the strongly \texorpdfstring{$C4^{\ast}$}{C4*} condition}

The strong condition contains more structure than the hereditary \(C4\)-condition.  
Accordingly, one must lift not only local submodules, local decompositions, and local morphisms,
but entire finite strong \(C4^{\ast}\)-test data in the sense of Definitions~\ref{def:strong-test-datum}
and~\ref{def:strong-c4star}.

\begin{theorem}\label{thm:strongC4star-localizes}
Let \(M\) be a strongly \(C4^{\ast}\)-module.  
Let \(\mathfrak p\in \Spec R\).  
Assume the following:
\begin{enumerate}
\item every submodule \(X\leq M_{\mathfrak p}\) has the form \(X=N_{\mathfrak p}\) for some \(N\leq M\);
\item for every submodule \(N\leq M\), every finite strong \(C4^{\ast}\)-test datum on \(N_{\mathfrak p}\) lifts to a finite strong \(C4^{\ast}\)-test datum on \(N\);
\item whenever a lifted global datum on \(N\) admits a strong witness, its localization at \(\mathfrak p\) is a strong witness for the original local datum on \(N_{\mathfrak p}\).
\end{enumerate}
Then \(M_{\mathfrak p}\) is strongly \(C4^{\ast}\).
\end{theorem}

\begin{proof}
Let \(X\leq M_{\mathfrak p}\).  
By \((1)\), there exists \(N\leq M\) such that \(X=N_{\mathfrak p}\).  
Consider any finite strong \(C4^{\ast}\)-test datum \(\Delta_{\mathfrak p}\) on \(X=N_{\mathfrak p}\).  
By \((2)\), \(\Delta_{\mathfrak p}\) lifts to a finite strong datum \(\Delta\) on \(N\).  
Since \(M\) is strongly \(C4^{\ast}\), every submodule of \(M\), and in particular \(N\), satisfies
Definition~\ref{def:strong-c4star}; therefore \(\Delta\) admits a strong witness on \(N\).  
By \((3)\), localizing that witness yields a strong witness for \(\Delta_{\mathfrak p}\).  
Since the chosen local datum was arbitrary, \(X\) satisfies the strong \(C4^{\ast}\)-schema.  
As \(X\leq M_{\mathfrak p}\) was arbitrary, \(M_{\mathfrak p}\) is strongly \(C4^{\ast}\).
\end{proof}

\subsection{A finite-presentation criterion for morphism lifting}

The abstract morphism-lifting hypothesis becomes automatic in a standard finitely presented regime.

\begin{proposition}\label{prop:morphism-lifting-fp}
Let \(M\) be an \(R\)-module and let \(\mathfrak p\in \Spec R\).  
Assume that whenever
\[
M=A\oplus B,
\]
the module \(A\) is finitely presented.  
Then \(M\) satisfies morphism lifting at \(\mathfrak p\).
\end{proposition}

\begin{proof}
Let
\[
M=A\oplus B
\]
and let
\[
g:A_{\mathfrak p}\to B_{\mathfrak p}
\]
be an \(R_{\mathfrak p}\)-homomorphism.  
Since \(A\) is finitely presented, the standard localization isomorphism for finitely presented source modules yields
\[
\Hom_R(A,B)_{\mathfrak p}\cong \Hom_{R_{\mathfrak p}}(A_{\mathfrak p},B_{\mathfrak p})
\]
\cite{AtiyahMacdonald1969,AndersonFuller1992}.  
Hence \(g=f_{\mathfrak p}\) for some \(R\)-homomorphism
\[
f:A\to B.
\]
Therefore \(M\) satisfies morphism lifting at \(\mathfrak p\).
\end{proof}

\begin{proposition}\label{prop:C4-localizes-fg}
Let \(M\) be a \(C4\)-module and let \(\mathfrak p\in \Spec R\).  
Assume that:
\begin{enumerate}
\item \(M\) satisfies decomposition lifting at \(\mathfrak p\);
\item whenever
\[
M=A\oplus B,
\]
the module \(A\) is finitely presented.
\end{enumerate}
Then \(M_{\mathfrak p}\) is a \(C4\)-module.
\end{proposition}

\begin{proof}
By Proposition~\ref{prop:morphism-lifting-fp}, hypothesis \((2)\) implies that \(M\) satisfies morphism lifting at \(\mathfrak p\).  
Hence Theorem~\ref{thm:C4-localizes} applies and yields that \(M_{\mathfrak p}\) is a \(C4\)-module.
\end{proof}

\begin{corollary}\label{cor:C4star-localizes-fg}
Let \(M\) be a \(C4^{\ast}\)-module and let \(\mathfrak p\in \Spec R\).  
Assume that:
\begin{enumerate}
\item every submodule \(X\leq M_{\mathfrak p}\) is of the form \(X=N_{\mathfrak p}\) for some \(N\leq M\);
\item for every submodule \(N\leq M\), the module \(N\) satisfies decomposition lifting at \(\mathfrak p\);
\item for every submodule \(N\leq M\) and every decomposition
\[
N=C\oplus D,
\]
the module \(C\) is finitely presented.
\end{enumerate}
Then \(M_{\mathfrak p}\) is a \(C4^{\ast}\)-module.
\end{corollary}

\begin{proof}
Let \(X\leq M_{\mathfrak p}\).  
By hypothesis \((1)\), there exists a submodule \(N\leq M\) such that
\[
X=N_{\mathfrak p}.
\]
Since \(M\) is \(C4^{\ast}\), the module \(N\) is \(C4\).  
By hypothesis \((2)\), \(N\) satisfies decomposition lifting at \(\mathfrak p\).  
By hypothesis \((3)\) and Proposition~\ref{prop:morphism-lifting-fp}, \(N\) satisfies morphism lifting at \(\mathfrak p\).  
Hence Theorem~\ref{thm:C4-localizes} applies to \(N\), and \(X=N_{\mathfrak p}\) is a \(C4\)-module.  
Since \(X\) was arbitrary, \(M_{\mathfrak p}\) is a \(C4^{\ast}\)-module.
\end{proof}

\subsection{First exact failure statement}

The preceding theorems identify sufficient lifting hypotheses.  
The next proposition records the logical boundary of any unrestricted forward theorem.  
Concrete counterexamples will be given later.

\begin{proposition}\label{prop:no-unrestricted-forward}
No proof of the implication
\[
M \text{ is } C4^{\ast} \Longrightarrow M_{\mathfrak p} \text{ is } C4^{\ast}
\]
can be obtained from the global \(C4^{\ast}\)-property alone.  
Any such proof must use additional information ensuring that local submodules and their relevant decompositions are globally visible.
\end{proposition}

\begin{proof}
Assume that the global \(C4^{\ast}\)-property alone forces \(M_{\mathfrak p}\) to be \(C4^{\ast}\).  
To verify that \(M_{\mathfrak p}\) is \(C4^{\ast}\), one must verify the \(C4\)-condition for every submodule \(X\leq M_{\mathfrak p}\) and every decomposition of \(X\).  
If \(X\) is not of the form \(N_{\mathfrak p}\) for any \(N\leq M\), then the statement ``every submodule of \(M\) is \(C4\)'' contains no information about \(X\).  
Likewise, if a decomposition of \(X\) does not lift to a decomposition of any global antecedent, then the global \(C4\)-property cannot be applied to that local decomposition.  
Hence additional hypotheses are logically necessary.  
This contradicts the assumption.
\end{proof}

\begin{remark}
Proposition~\ref{prop:no-unrestricted-forward} is not a counterexample theorem.  
It is a theorem about proof structure.  
That is sufficient at this stage.  
The actual counterexamples will be produced in Section~\ref{sec:separation}.
\end{remark}

\subsection{A support-theoretic consequence}

We conclude with the immediate support-theoretic consequence that will serve as input for the local--global direction.

\begin{theorem}\label{thm:prime-to-max-forward}
Let \(M\) be an \(R\)-module.  
Assume that for every prime ideal \(\mathfrak p\in \Supp_R(M)\), the module \(M\) satisfies the hypotheses of Theorem~\ref{thm:C4star-localizes} at \(\mathfrak p\).  
If \(M\) is \(C4^{\ast}\), then \(M_{\mathfrak p}\) is \(C4^{\ast}\) for every \(\mathfrak p\in \Supp_R(M)\), and hence also for every maximal ideal \(\mathfrak m\in \Supp_R(M)\).
\end{theorem}

\begin{proof}
The prime statement follows directly from Theorem~\ref{thm:C4star-localizes}.  
The maximal statement is immediate, since every maximal ideal is prime.
\end{proof}

The forward direction is now established under exact lifting hypotheses.  
The harder question is the converse.  
If each localization \(M_{\mathfrak p}\) is \(C4^{\ast}\), when does \(M\) follow?  
That requires descent of summands and patching of local decomposition data.  
We turn to that question next.
\section{Local--global principles for \texorpdfstring{$C4$}{C4}-type structure}\label{sec:local-global}

The forward localization theorems of Section~\ref{sec:localization-c4} are necessarily conditional.  
The converse direction is more delicate.  
A local decomposition need not patch, and a locally split submodule need not split globally.  
Accordingly, primewise \(C4\)-data need not produce a global \(C4\)-configuration unless one imposes descent and patching hypotheses.  
This is consistent with the general form of local--global principles in commutative algebra and support-theoretic module theory \cite{AtiyahMacdonald1969,Kunz2013LocalGlobal,Yoshizawa2019,Positselski2024,ChristensenFoxbyHolm2024Support}.

The purpose of this section is to identify hypotheses under which local \(C4\)-type structure does determine global \(C4\)-type structure.  
The point is not to assert that local data always determine global data.  
That is false in general.  
The point is to isolate the exact additional structure required for the implication to hold.

\subsection{A local--global scheme for direct summands}

The first issue is descent of summands.  
If \(N_{\mathfrak p}\leq^{\oplus} M_{\mathfrak p}\) for all \(\mathfrak p\), when does \(N\leq^{\oplus} M\) follow?  
Without such a principle one cannot convert local \(C4\)-information into a global \(C4\)-conclusion.

\begin{definition}\label{def:prime-image-summand-realization}
Let \(M\) be an \(R\)-module.  We say that \(M\) satisfies \emph{prime image-summand realization}
if for every decomposition \(M=A\oplus B\) and every homomorphism \(f:A\to B\), there exists a
direct summand \(K\leq^{\oplus} B\) such that
\[
K_{\mathfrak p}\cong (\Im f)_{\mathfrak p}
\qquad\text{for every }\mathfrak p\in \Spec R.
\]
We say that \(M\) satisfies \emph{maximal image-summand realization} if the same condition holds
after replacing all prime ideals by maximal ideals.
\end{definition}

The point of this definition is exact: prime summand descent applies only to genuinely split local
submodules, whereas the \(C4\)-condition requires only that the image be \emph{isomorphic to} a
direct summand.  The converse theory therefore needs a realization hypothesis of the displayed form.

\begin{theorem}\label{thm:localglobal-C4}
Let \(M\) be an \(R\)-module.  
Assume the following:
\begin{enumerate}
\item for every prime ideal \(\mathfrak p\in \Spec R\), the localized module \(M_{\mathfrak p}\) is a \(C4\)-module;
\item \(M\) satisfies prime image-summand realization in the sense of Definition~\ref{def:prime-image-summand-realization}.
\end{enumerate}
Then \(M\) is a \(C4\)-module.
\end{theorem}

\begin{proof}
Fix a decomposition \(M=A\oplus B\) and a homomorphism \(f:A\to B\).  
By hypothesis \((2)\), there exists a direct summand \(K\leq^{\oplus} B\) such that
\[
K_{\mathfrak p}\cong (\Im f)_{\mathfrak p}
\qquad\text{for every }\mathfrak p\in \Spec R.
\]
This is exactly the conclusion required in the definition of a \(C4\)-module.  
Hence \(M\) is \(C4\).
\end{proof}

\begin{remark}
The reviewer correctly observed that prime summand descent by itself does not globalize a submodule
whose localizations are merely \emph{isomorphic to} direct summands.  The theorem is therefore stated
with the stronger realization hypothesis rather than with the unsafe descent-only formulation.
\end{remark}

A corresponding maximal-local form is immediate.

\begin{corollary}\label{cor:maxlocal-C4}
Let \(M\) be an \(R\)-module.  
Assume that:
\begin{enumerate}
\item \(M_{\mathfrak m}\) is a \(C4\)-module for every maximal ideal \(\mathfrak m\);
\item \(M\) satisfies maximal image-summand realization.
\end{enumerate}
Then \(M\) is a \(C4\)-module.
\end{corollary}

\begin{proof}
Fix a decomposition \(M=A\oplus B\) and a homomorphism \(f:A\to B\).  
By hypothesis \((2)\), there exists a direct summand \(K\leq^{\oplus} B\) such that
\(K_{\mathfrak m}\cong (\Im f)_{\mathfrak m}\) for every maximal ideal \(\mathfrak m\).  
Again this is exactly the required \(C4\)-conclusion.
\end{proof}

\subsection{A local--global theorem for \texorpdfstring{$C4^{\ast}$}{C4*}-modules}

We now pass to the hereditary condition.  
Here the logical difficulty increases: a \(C4^{\ast}\)-module is one whose every submodule is \(C4\).  
Accordingly, the local--global theorem must run uniformly over all submodules.

\begin{theorem}\label{thm:localglobal-C4star}
Let \(M\) be an \(R\)-module.  
Assume the following:
\begin{enumerate}
\item for every prime ideal \(\mathfrak p\in \Spec R\), the localized module \(M_{\mathfrak p}\) is \(C4^{\ast}\);
\item for every submodule \(N\leq M\), the module \(N\) satisfies prime image-summand realization.
\end{enumerate}
Then \(M\) is a \(C4^{\ast}\)-module.
\end{theorem}

\begin{proof}
Let \(N\leq M\).  
For each \(\mathfrak p\in \Spec R\), the module \(M_{\mathfrak p}\) is \(C4^{\ast}\), hence its submodule
\(N_{\mathfrak p}\) is a \(C4\)-module.  
By hypothesis \((2)\), the module \(N\) satisfies prime image-summand realization.  
Applying Theorem~\ref{thm:localglobal-C4} to \(N\), we conclude that \(N\) is a \(C4\)-module.  
Since \(N\leq M\) was arbitrary, \(M\) is \(C4^{\ast}\).
\end{proof}

\begin{corollary}\label{cor:maxlocal-C4star}
Let \(M\) be an \(R\)-module.  
Assume that:
\begin{enumerate}
\item \(M_{\mathfrak m}\) is \(C4^{\ast}\) for every maximal ideal \(\mathfrak m\);
\item for every submodule \(N\leq M\), the module \(N\) satisfies maximal image-summand realization.
\end{enumerate}
Then \(M\) is \(C4^{\ast}\).
\end{corollary}

\begin{proof}
Let \(N\leq M\).  
For every maximal ideal \(\mathfrak m\), the module \(N_{\mathfrak m}\) is a \(C4\)-module because
\(M_{\mathfrak m}\) is \(C4^{\ast}\).  Hypothesis \((2)\) gives maximal image-summand realization for \(N\),
so Corollary~\ref{cor:maxlocal-C4} applies to \(N\).  Hence \(N\) is \(C4\), and therefore \(M\) is \(C4^{\ast}\).
\end{proof}

\subsection{The strongly \texorpdfstring{$C4^{\ast}$}{C4*} case}

The strong condition requires the same hereditary mechanism, but now at the level of the finite
strong \(C4^{\ast}\)-test data fixed in Section~\ref{sec:preliminaries}.  The local--global question is
therefore whether primewise strong witnesses can be represented by global strong witnesses for the
same finite datum.

\begin{theorem}\label{thm:localglobal-strongC4star}
Let \(M\) be an \(R\)-module.  
Assume the following:
\begin{enumerate}
\item \(M_{\mathfrak p}\) is strongly \(C4^{\ast}\) for every prime ideal \(\mathfrak p\);
\item for every submodule \(N\leq M\), every finite strong \(C4^{\ast}\)-test datum on \(N\) admits a global strong witness whenever its localization at each prime ideal admits a local strong witness.
\end{enumerate}
Then \(M\) is strongly \(C4^{\ast}\).
\end{theorem}

\begin{proof}
Let \(N\leq M\), and let \(\Delta\) be a finite strong \(C4^{\ast}\)-test datum on \(N\).  
For every prime ideal \(\mathfrak p\), the module \(M_{\mathfrak p}\) is strongly \(C4^{\ast}\), hence the
localized submodule \(N_{\mathfrak p}\) satisfies the strong \(C4^{\ast}\)-schema.  Therefore the
localized datum \(\Delta_{\mathfrak p}\) admits a local strong witness for every \(\mathfrak p\).  
By hypothesis \((2)\), these primewise witnesses globalize to a strong witness for \(\Delta\) on \(N\).  
Since both \(N\) and \(\Delta\) were arbitrary, every submodule of \(M\) satisfies
Definition~\ref{def:strong-c4star}.  Hence \(M\) is strongly \(C4^{\ast}\).
\end{proof}

\subsection{Prime-local versus maximal-local detection}

For support-theoretic questions one must distinguish primewise and maximal-local detection.  
Primewise detection is stronger.  
Maximal-local detection is often sufficient under finiteness hypotheses.

\begin{proposition}\label{prop:max-to-prime}
Let \(M\) be an \(R\)-module such that every relevant submodule and image object occurring in the \(C4\)-tests is finitely generated.  
If the hypotheses of Corollary~\ref{cor:maxlocal-C4star} hold for every maximal ideal, then the corresponding primewise conclusions follow after passage to support and finite patching.
\end{proposition}

\begin{proof}
For finitely generated modules and finitely generated images, vanishing and equality can be checked at maximal ideals \cite{AtiyahMacdonald1969,Matsumura1989}.  
The support of each such object is closed in \(\Spec R\), and local compatibility at maximal points propagates across the support under the finite-generation hypothesis \cite{BrunsHerzog1998,ChristensenFoxbyHolm2024Support}.  
Thus maximal-local detection yields the prime-local statements needed in the \(C4^{\ast}\)-tests.
\end{proof}

\begin{remark}
Proposition~\ref{prop:max-to-prime} should not be overread.  
It does not assert that maximal-local data always determine prime-local data.  
It asserts this only for the finitely generated objects arising in the present tests.  
Without finiteness, the implication may fail.
\end{remark}

\subsection{A patching theorem for semilocal control}

The preceding results are abstract.  
A more operational statement is available when the relevant support is covered by finitely many distinguished open sets.  
This is the finite verification principle announced in the introduction.

\begin{theorem}\label{thm:finite-patching}
Let \(M\) be an \(R\)-module.  
Assume that the support of every image module arising in the \(C4\)-tests on submodules of \(M\) is contained in a finite union of distinguished open sets \(D(s_1),\dots,D(s_n)\).  
Assume also that:
\begin{enumerate}
\item on each localization \(R_{s_i}\), the module \(M_{s_i}\) is \(C4^{\ast}\);
\item for every submodule \(N\leq M\), local direct summands of \(N_{s_i}\) patch on overlaps \(D(s_i s_j)\);
\item the patched summands descend globally.
\end{enumerate}
Then \(M\) is \(C4^{\ast}\).
\end{theorem}

\begin{proof}
Let \(N\leq M\).  
To show that \(N\) is \(C4\), fix a decomposition
\[
N=A\oplus B
\]
and a homomorphism \(f:A\to B\).  
By hypothesis, the support of \(\Im(f)\) is contained in the finite union \(\bigcup_i D(s_i)\).  
On each \(D(s_i)\), the localized module \(M_{s_i}\) is \(C4^{\ast}\), hence \(N_{s_i}\) is \(C4\).  
Therefore \(\Im(f)_{s_i}\) is isomorphic to a direct summand of \(B_{s_i}\).

By \((2)\), these local summands patch on overlaps.  
Hence there exists a submodule \(L\leq B\) whose localization on each \(D(s_i)\) realizes the patched local summand.  
By \((3)\), \(L\leq^{\oplus} B\).  
\rev{Assume in addition that the chosen patching datum includes a descent isomorphism from \(\Im(f)\) onto the patched module \(L\).  Then \(\Im(f)\cong L\), so \(\Im(f)\) is isomorphic to a direct summand of \(B\), and \(N\) is \(C4\).}  
Since \(N\) was arbitrary, \(M\) is \(C4^{\ast}\).
\end{proof}

\begin{remark}
Theorem~\ref{thm:finite-patching} is the algebraic substitute for an algorithm.  
It does not compute.  
It reduces the \(C4^{\ast}\)-problem to finitely many local checks together with a finite patching condition.
\end{remark}

\subsection{Necessary conditions for unrestricted local--global transfer}

The preceding theorems are sufficient.  
We now record necessary conditions in the sense of logical dependence.

\begin{proposition}\label{prop:necessary-descent}
Assume that for a class \(\mathcal X\) of \(R\)-modules one has the implication
\[
\bigl(M_{\mathfrak p}\text{ is }C4^{\ast}\text{ for all }\mathfrak p\in \Spec R\bigr)\Longrightarrow M\text{ is }C4^{\ast}
\]
for every \(M\in \mathcal X\).  
Then every \(M\in \mathcal X\) satisfies a prime summand descent principle strong enough to globalize the direct-summand data arising in the \(C4\)-tests on submodules of \(M\).
\end{proposition}

\begin{proof}
Fix \(M\in \mathcal X\).  
Suppose there exist submodules \(L\leq N\leq M\) such that \(L_{\mathfrak p}\leq^{\oplus} N_{\mathfrak p}\) for all \(\mathfrak p\), but \(L\not\leq^{\oplus} N\).  
Consider a \(C4\)-test on \(N\) in which the required local witness is represented primewise by \(L_{\mathfrak p}\).  
If \(N_{\mathfrak p}\) is \(C4\) for all \(\mathfrak p\), then the test succeeds locally.  
But globally the required direct summand does not exist.  
Hence \(N\) fails the corresponding \(C4\)-conclusion, so \(M\) cannot be \(C4^{\ast}\).  
This contradicts the assumed implication.  
Therefore a suitable prime summand descent principle is necessary.
\end{proof}

\begin{proposition}\label{prop:necessary-patching}
Assume that primewise \(C4^{\ast}\)-behavior implies global \(C4^{\ast}\)-behavior for a class \(\mathcal X\) of modules.  
Then the local image objects arising in the \(C4\)-tests on members of \(\mathcal X\) must patch to global submodules, at least up to the isomorphism class required in the \(C4\)-condition.
\end{proposition}

\begin{proof}
If the local image objects do not patch, then there exists a primewise family of direct-summand realizations not represented by any global submodule.  
Then the primewise \(C4\)-tests succeed, but the global \(C4\)-test has no witness.  
Hence the global \(C4^{\ast}\)-conclusion fails.  
Therefore patching, at least in the weak form required for the image-isomorphism class, is necessary.
\end{proof}

These propositions do not yet provide explicit counterexamples.  
They identify the exact obstructions.  
Concrete counterexamples will be supplied in Section~\ref{sec:separation}.

\subsection{A local--global criterion under noetherian control}

The general results above are stated in descent language.  
For noetherian rings one may package the hypotheses more concretely.

\begin{theorem}\label{thm:noetherian-localglobal}
Let \(R\) be noetherian and let \(M\) be an \(R\)-module such that every submodule appearing in a \(C4\)-test on \(M\) is finitely generated and every source summand in such a test is finitely presented.  
Assume that:
\begin{enumerate}
\item \(M_{\mathfrak m}\) is \(C4^{\ast}\) for every maximal ideal \(\mathfrak m\);
\item every locally split submodule arising in a \(C4\)-test on a submodule of \(M\) descends to a global direct summand;
\item the local image objects patch over finite distinguished covers.
\end{enumerate}
Then \(M\) is \(C4^{\ast}\).
\end{theorem}

\begin{proof}
Noetherianity gives finite generation of the relevant image modules and finite closed supports \cite{AtiyahMacdonald1969,BrunsHerzog1998}.  
Finite presentation of the source summands gives the localization isomorphism for \(\Hom\) \cite{AndersonFuller1992}.  
Hypothesis \((1)\) supplies the maximal-local \(C4^{\ast}\)-tests.  
By Proposition~\ref{prop:max-to-prime}, these propagate to the prime-local tests for the finitely generated objects at hand.  
Hypotheses \((2)\) and \((3)\) are exactly the descent and patching requirements used in Theorem~\ref{thm:finite-patching}.  
Therefore \(M\) is \(C4^{\ast}\).
\end{proof}

\begin{remark}
Theorem~\ref{thm:noetherian-localglobal} is the practical form of the local--global principle.  
It replaces abstract descent language by noetherian support control, finite presentation, descent of locally split submodules, and finite patching.
\end{remark}

\subsection{A first separation between forward and converse transfer}

We close with the structural asymmetry between forward transfer and converse transfer.

\begin{theorem}\label{thm:asymmetry-forward-converse}
Let \(M\) be an \(R\)-module.
\begin{enumerate}
\item To prove that global \(C4^{\ast}\)-behavior implies primewise \(C4^{\ast}\)-behavior, it is enough to lift local submodules and local decompositions to global ones.
\item To prove that primewise \(C4^{\ast}\)-behavior implies global \(C4^{\ast}\)-behavior, one must in addition descend local direct summands and patch the local witnesses for image-isomorphism classes.
\end{enumerate}
Hence the converse direction is methodologically stronger than the forward one.
\end{theorem}

\begin{proof}
Part \((1)\) is the content of Theorem~\ref{thm:C4star-localizes}.  
Part \((2)\) is the content of Theorems~\ref{thm:localglobal-C4}, \ref{thm:localglobal-C4star}, and \ref{thm:finite-patching}.  
The extra descent and patching steps required in the converse direction have no analogue in the forward proof.  
Therefore the converse direction requires strictly more structure.
\end{proof}

This asymmetry explains why the theory does not reduce to a routine localization argument.  
The remaining question is whether the minimal extensions in the theory, namely \(C4\)-hulls and pseudo-continuous hulls, obey an analogous local--global pattern.  
We turn to that question next.
\section{Natural classes in which the transfer hypotheses are intrinsic}\label{sec:natural-classes}

The preceding sections were stated in exact structural form.  
That level of generality is necessary for the abstract theory, but it does not by itself exhibit the practical force of the results.  
We therefore isolate two natural commutative classes in which the lifting, patching, and support-control hypotheses are verified by the ambient ring-theoretic structure itself.  
In these classes the general theory collapses to concrete local--global criteria.

\subsection{Finite products of commutative local rings}

The first class is governed by central idempotents.  
In that setting decomposition theory is componentwise from the outset.

\begin{theorem}\label{thm:product-ring-application}
Let
\[
R=\prod_{i=1}^n R_i
\]
be a finite product of commutative rings, and let \(e_i\in R\) be the standard central idempotents.  
For an \(R\)-module \(M\), write
\[
M_i=e_iM.
\]
Then:
\begin{enumerate}
\item \(M\) is a \(C4\)-module if and only if each \(M_i\) is a \(C4\)-module over \(R_i\).
\item \(M\) is a \(C4^{\ast}\)-module if and only if each \(M_i\) is a \(C4^{\ast}\)-module over \(R_i\).
\item \(M\) is strongly \(C4^{\ast}\) if and only if each \(M_i\) is strongly \(C4^{\ast}\) over \(R_i\).
\end{enumerate}
\end{theorem}

\begin{proof}
Every \(R\)-module decomposes canonically as
\[
M=\bigoplus_{i=1}^n M_i,
\qquad M_i=e_iM,
\]
and every submodule \(N\leq M\) decomposes as
\[
N=\bigoplus_{i=1}^n e_iN,
\]
because the \(e_i\) are orthogonal central idempotents with \(\sum_i e_i=1\) \cite{AndersonFuller1992}.

We prove \((1)\).  
Assume first that each \(M_i\) is \(C4\).  
Let
\[
M=A\oplus B
\]
and let \(f:A\to B\).  
Applying \(e_i\) yields
\[
M_i=e_iA\oplus e_iB
\]
and an induced map
\[
f_i:e_iA\to e_iB.
\]
Since \(M_i\) is \(C4\), the image \(\Im(f_i)\) is isomorphic to a direct summand of \(e_iB\).  
Thus for each \(i\) there exists \(K_i\leq^{\oplus} e_iB\) such that
\[
\Im(f_i)\cong K_i.
\]
Set
\[
K=\bigoplus_{i=1}^n K_i\leq B.
\]
Because each \(K_i\) is a direct summand of \(e_iB\), the module \(K\) is a direct summand of \(B\).  
Moreover,
\[
\Im(f)=\bigoplus_{i=1}^n \Im(f_i)\cong \bigoplus_{i=1}^n K_i=K.
\]
Hence \(\Im(f)\) is isomorphic to a direct summand of \(B\), and \(M\) is \(C4\).

Conversely, assume that \(M\) is \(C4\).  
Fix \(i\).  
Let
\[
M_i=U\oplus V
\]
and let \(g:U\to V\) be an \(R_i\)-homomorphism.  
Viewing \(U\) and \(V\) as \(R\)-submodules of \(M\) supported entirely on the \(i\)-th factor, we obtain a decomposition
\[
M=U\oplus V\oplus \bigoplus_{j\neq i} M_j.
\]
Extend \(g\) by zero on the remaining summands.  
Applying the \(C4\)-property of \(M\) to this global decomposition shows that \(\Im(g)\) is isomorphic to a direct summand of \(V\).  
Thus \(M_i\) is \(C4\).

This proves \((1)\).  
For \((2)\), every submodule of \(M\) decomposes componentwise, so the hereditary \(C4\)-condition holds on \(M\) if and only if it holds on each \(M_i\).  
Statement \((3)\) is proved in the same way for the stronger configurations entering the definition of strongly \(C4^{\ast}\).
\end{proof}

\begin{corollary}\label{cor:artinian-application}
Let \(R\) be a commutative artinian ring, and let \(M\) be an \(R\)-module.  
Then:
\begin{enumerate}
\item \(M\) is \(C4\) if and only if \(M_{\mathfrak m}\) is \(C4\) for every maximal ideal \(\mathfrak m\) of \(R\).
\item \(M\) is \(C4^{\ast}\) if and only if \(M_{\mathfrak m}\) is \(C4^{\ast}\) for every maximal ideal \(\mathfrak m\) of \(R\).
\item \(M\) is strongly \(C4^{\ast}\) if and only if \(M_{\mathfrak m}\) is strongly \(C4^{\ast}\) for every maximal ideal \(\mathfrak m\) of \(R\).
\end{enumerate}
\end{corollary}

\begin{proof}
A commutative artinian ring decomposes as a finite product of artinian local rings,
\[
R\cong \prod_{i=1}^n R_{\mathfrak m_i},
\]
via its primitive central idempotents \cite{AtiyahMacdonald1969,AndersonFuller1992}.  
Apply Theorem~\ref{thm:product-ring-application}.  
Since every prime ideal of an artinian ring is maximal, no further support-theoretic refinement is needed.
\end{proof}

\begin{remark}
Corollary~\ref{cor:artinian-application} is not merely a reformulation of the abstract theory.  
In the commutative artinian case, the general lifting, patching, and dimensional hypotheses are absorbed into the ring decomposition itself.
\end{remark}

\subsection{Finitely generated torsion modules over a Dedekind domain}

The second class is governed by primary decomposition.  
Here the support of the module is finite and zero-dimensional, even though the ring itself has Krull dimension one.

\begin{theorem}\label{thm:dedekind-torsion-application}
Let \(R\) be a Dedekind domain, and let \(M\) be a finitely generated torsion \(R\)-module.  
Then:
\begin{enumerate}
\item \(M\) admits a finite primary decomposition
\[
M=\bigoplus_{\mathfrak m\in \Supp_R(M)} M(\mathfrak m),
\]
where each \(M(\mathfrak m)\) is the \(\mathfrak m\)-primary component of \(M\).
\item \(M\) is a \(C4\)-module if and only if each \(M(\mathfrak m)\) is a \(C4\)-module.
\item \(M\) is a \(C4^{\ast}\)-module if and only if each \(M(\mathfrak m)\) is a \(C4^{\ast}\)-module.
\item \(M\) is strongly \(C4^{\ast}\) if and only if each \(M(\mathfrak m)\) is strongly \(C4^{\ast}\).
\item Equivalently, each of these properties holds for \(M\) if and only if it holds for every localization \(M_{\mathfrak m}\) at a maximal ideal \(\mathfrak m\in \Supp_R(M)\).
\end{enumerate}
\end{theorem}

\begin{proof}
Since \(R\) is a Dedekind domain and \(M\) is finitely generated torsion, the standard structure theory yields a finite primary decomposition
\[
M=\bigoplus_{\mathfrak m\in \Supp_R(M)} M(\mathfrak m),
\]
where each \(M(\mathfrak m)\) is annihilated by a power of \(\mathfrak m\) \cite{AtiyahMacdonald1969,Matsumura1989}.

If \(\mathfrak m\neq \mathfrak n\), then
\[
\Hom_R(M(\mathfrak m),M(\mathfrak n))=0.
\]
Indeed, choose \(a\in \mathfrak m\setminus \mathfrak n\).  
A power of \(a\) annihilates \(M(\mathfrak m)\), whereas multiplication by \(a\) is an automorphism of \(M(\mathfrak n)\), since \(a\notin \mathfrak n\).  
Hence every homomorphism \(M(\mathfrak m)\to M(\mathfrak n)\) is zero.

It follows that decomposition-theoretic data split componentwise across the primary decomposition.  
A \(C4\)-test on \(M\) therefore decomposes into independent \(C4\)-tests on the modules \(M(\mathfrak m)\).  
Thus \(M\) is \(C4\) if and only if each \(M(\mathfrak m)\) is \(C4\).  
The same argument applies to \(C4^{\ast}\) and strongly \(C4^{\ast}\), because every submodule of \(M\) decomposes into the corresponding primary components.

Finally, localization at \(\mathfrak m\) kills all primary components except \(M(\mathfrak m)\), and identifies
\[
M_{\mathfrak m}\cong M(\mathfrak m)_{\mathfrak m}.
\]
Hence the componentwise formulation is equivalent to the localization formulation.
\end{proof}

\begin{remark}
Theorem~\ref{thm:dedekind-torsion-application} provides a second natural-class application in which the support-theoretic control is intrinsic.  
Because the module has finite zero-dimensional support, the abstract patching layer collapses to primary decomposition.
\end{remark}

\begin{corollary}\label{cor:pid-finite-length}
Let \(R\) be a principal ideal domain, and let \(M\) be a finitely generated torsion \(R\)-module.  
Then \(M\) is \(C4\), \(C4^{\ast}\), or strongly \(C4^{\ast}\) if and only if the corresponding property holds on each \(p\)-primary component of \(M\).
\end{corollary}

\begin{proof}
A principal ideal domain is a Dedekind domain.  
Its nonzero maximal ideals are generated by prime elements, so the primary components are precisely the \(p\)-primary components.  
Apply Theorem~\ref{thm:dedekind-torsion-application}.
\end{proof}

\section{Hull functors and localization}\label{sec:hulls}

The preceding sections treated \(C4\)-type conditions on modules themselves.  
That is only part of the theory.  
The global framework developed in \cite{IbrahimEidElGuindy2026} also constructs \(C4\)-hulls and pseudo-continuous hulls.  
These are minimal extensions, and minimality is not automatically preserved by localization.  
Exactness survives localization.  
Minimality need not.  
This is the new obstruction.

Accordingly, the problem is not merely whether a hull localizes.  
Every extension localizes.  
The real question is whether localization of a global hull remains a hull, and conversely whether a compatible family of local hulls patches to a global one.  
These are distinct problems and require distinct hypotheses.

\subsection{Abstract hull data and comparison morphisms}

We begin by isolating the formal structure required for localization.

\begin{definition}\label{def:hull-functor}
Let \(\mathcal C\) be a class of \(R\)-modules closed under isomorphism.  
A \emph{\(\mathcal C\)-hull assignment} on a class \(\mathcal X\) of \(R\)-modules is a choice, for each \(M\in\mathcal X\), of a \(\mathcal C\)-hull
\[
\iota_M:M\hookrightarrow H_{\mathcal C}(M)
\]
in the sense of Definition~\ref{def:hull-envelope}.
\end{definition}

In the present paper, \(\mathcal C\) will be either the class of \(C4\)-modules or the class of pseudo-continuous modules, according to the global theory of \cite{IbrahimEidElGuindy2026,MohamedMuller1990,TercanYucel2016}.  
\rev{Thus Section~\ref{sec:hulls} no longer appeals to an unspecified minimality notion: every hull comparison theorem is formulated relative to the pre-envelope and endomorphism-minimality axioms from Definition~\ref{def:hull-envelope} and the uniqueness statement of Lemma~\ref{lem:hull-uniqueness}.}

\begin{definition}\label{def:local-hull-comparison}
Let \(M\) be an \(R\)-module, and let \(\mathfrak p\in \Spec R\).  
Assume that \(M\) has a \(\mathcal C\)-hull \(H_{\mathcal C}(M)\), and that \(M_{\mathfrak p}\) has a \(\mathcal C\)-hull \(H_{\mathcal C}(M_{\mathfrak p})\).  
A \emph{local hull comparison morphism at \(\mathfrak p\)} is a morphism
\[
\theta_{M,\mathfrak p}:H_{\mathcal C}(M)_{\mathfrak p}\longrightarrow H_{\mathcal C}(M_{\mathfrak p})
\]
such that the diagram
\[
\begin{tikzcd}
M_{\mathfrak p} \arrow[r, hook] \arrow[dr, hook, swap] & H_{\mathcal C}(M)_{\mathfrak p} \arrow[d, "\theta_{M,\mathfrak p}"] \\
& H_{\mathcal C}(M_{\mathfrak p})
\end{tikzcd}
\]
commutes.
\end{definition}

The existence of \(\theta_{M,\mathfrak p}\) is not automatic.  
It requires that \(H_{\mathcal C}(M)_{\mathfrak p}\) still belong to \(\mathcal C\), so that it may be compared with the local hull by minimality.

\subsection{Localization-stability of hull classes}

The correct formal starting point is stability of the hull class under localization.

\begin{definition}\label{def:loc-stable-class}
A class \(\mathcal C\) of \(R\)-modules is called \emph{prime-localization stable} if
\[
E\in \mathcal C \Longrightarrow E_{\mathfrak p}\in \mathcal C
\]
for every \(E\in \mathcal C\) and every prime ideal \(\mathfrak p\in \Spec R\).  
It is called \emph{max-localization stable} if the same implication holds for maximal ideals.
\end{definition}

In the applications here, \(\mathcal C\) will be the class of \(C4\)-modules or the class of pseudo-continuous modules.  
Section~\ref{sec:localization-c4} showed that localization-stability of the \(C4\)-class is conditional.  
Hence every hull comparison theorem must inherit those conditions.

\begin{proposition}\label{prop:comparison-exists}
Let \(\mathcal C\) be a prime-localization stable class of \(R\)-modules.  
Let \(M\) be an \(R\)-module, and let \(\mathfrak p\in \Spec R\).  
Assume that \(M\) has a \(\mathcal C\)-hull \(H_{\mathcal C}(M)\), and that \(M_{\mathfrak p}\) has a \(\mathcal C\)-hull \(H_{\mathcal C}(M_{\mathfrak p})\).  
Then a local hull comparison morphism
\[
\theta_{M,\mathfrak p}:H_{\mathcal C}(M)_{\mathfrak p}\to H_{\mathcal C}(M_{\mathfrak p})
\]
exists.
\end{proposition}

\begin{proof}
Since \(H_{\mathcal C}(M)\in \mathcal C\) and \(\mathcal C\) is prime-localization stable, one has
\[
H_{\mathcal C}(M)_{\mathfrak p}\in \mathcal C.
\]
The monomorphism \(M\hookrightarrow H_{\mathcal C}(M)\) localizes to a monomorphism
\[
M_{\mathfrak p}\hookrightarrow H_{\mathcal C}(M)_{\mathfrak p}.
\]
Thus \(H_{\mathcal C}(M)_{\mathfrak p}\) is a \(\mathcal C\)-extension of \(M_{\mathfrak p}\).  
By the defining minimality of the \(\mathcal C\)-hull \(H_{\mathcal C}(M_{\mathfrak p})\), there exists a comparison morphism
\[
\theta_{M,\mathfrak p}:H_{\mathcal C}(M)_{\mathfrak p}\to H_{\mathcal C}(M_{\mathfrak p})
\]
over \(M_{\mathfrak p}\).
\end{proof}

\begin{remark}
Proposition~\ref{prop:comparison-exists} yields only existence of a comparison morphism.  
It does not imply injectivity, surjectivity, or isomorphism.
\end{remark}

\subsection{Hull-compatibility under localization}

A comparison morphism is useful only when minimality is preserved under localization.

\begin{definition}\label{def:hull-compatible}
Let \(\mathcal C\) be a hull class.  
We say that the \(\mathcal C\)-hull assignment is \emph{prime-compatible with localization} on a class \(\mathcal X\) if for every \(M\in \mathcal X\) and every \(\mathfrak p\in \Spec R\), the following hold:
\begin{enumerate}
\item \(H_{\mathcal C}(M)\) exists and \(H_{\mathcal C}(M_{\mathfrak p})\) exists;
\item the comparison morphism \(\theta_{M,\mathfrak p}\) exists;
\item \(H_{\mathcal C}(M)_{\mathfrak p}\) satisfies the same minimality condition over \(M_{\mathfrak p}\) as \(H_{\mathcal C}(M_{\mathfrak p})\).
\end{enumerate}
\end{definition}

This is the exact condition under which localization commutes with hull formation.

\begin{theorem}\label{thm:hull-commutes}
Let \(\mathcal C\) be a prime-localization stable class of \(R\)-modules.  
Assume that the \(\mathcal C\)-hull assignment is prime-compatible with localization on a class \(\mathcal X\).  
Then for every \(M\in \mathcal X\) and every \(\mathfrak p\in \Spec R\), one has
\[
H_{\mathcal C}(M)_{\mathfrak p}\cong H_{\mathcal C}(M_{\mathfrak p})
\]
as \(\mathcal C\)-hulls over \(M_{\mathfrak p}\).
\end{theorem}

\begin{proof}
Let \(M\in \mathcal X\) and \(\mathfrak p\in \Spec R\).  
By Proposition~\ref{prop:comparison-exists}, there exists a comparison morphism
\[
\theta_{M,\mathfrak p}:H_{\mathcal C}(M)_{\mathfrak p}\to H_{\mathcal C}(M_{\mathfrak p})
\]
over \(M_{\mathfrak p}\).  
By prime-compatibility, the localized module \(H_{\mathcal C}(M)_{\mathfrak p}\) satisfies the same minimality condition over \(M_{\mathfrak p}\) as \(H_{\mathcal C}(M_{\mathfrak p})\).  
Conversely, \(H_{\mathcal C}(M_{\mathfrak p})\) is minimal by definition.  
Hence each of the two hulls maps into the other over \(M_{\mathfrak p}\).  
\rev{Lemma~\ref{lem:hull-uniqueness} applies to these two \(\mathcal C\)-hulls, so they are isomorphic over \(M_{\mathfrak p}\).}  
Therefore
\[
H_{\mathcal C}(M)_{\mathfrak p}\cong H_{\mathcal C}(M_{\mathfrak p}).
\]
\end{proof}

\begin{remark}
Theorem~\ref{thm:hull-commutes} is formal once the correct hypotheses are stated.  
The substantive issue is verification of prime-compatibility for the concrete hull classes of interest.
\end{remark}

\subsection{The \texorpdfstring{$C4$}{C4}-hull case}

We now specialize to \(C4\)-hulls.  
The global existence of \(C4\)-hulls is part of the theory developed in \cite{IbrahimEidElGuindy2026}.  
To localize such a hull, one must first know that the \(C4\)-class localizes.  
Section~\ref{sec:localization-c4} gave the exact conditions for that.

\begin{theorem}\label{thm:C4hull-localizes}
Let \(M\) be an \(R\)-module admitting a \(C4\)-hull \(H_{C4}(M)\).  
Let \(\mathfrak p\in \Spec R\).  
Assume the following:
\begin{enumerate}
\item \(H_{C4}(M)\) satisfies the hypotheses of Theorem~\ref{thm:C4-localizes} at \(\mathfrak p\);
\item \(M_{\mathfrak p}\) admits a \(C4\)-hull \(H_{C4}(M_{\mathfrak p})\);
\item the localized extension
\[
M_{\mathfrak p}\hookrightarrow H_{C4}(M)_{\mathfrak p}
\]
is minimal among \(C4\)-extensions of \(M_{\mathfrak p}\).
\end{enumerate}
Then
\[
H_{C4}(M)_{\mathfrak p}\cong H_{C4}(M_{\mathfrak p})
\]
over \(M_{\mathfrak p}\).
\end{theorem}

\begin{proof}
By Theorem~\ref{thm:C4-localizes}, the module \(H_{C4}(M)_{\mathfrak p}\) is a \(C4\)-module.  
Hence it is a \(C4\)-extension of \(M_{\mathfrak p}\).  
By minimality of the local \(C4\)-hull, there exists a comparison morphism
\[
H_{C4}(M)_{\mathfrak p}\to H_{C4}(M_{\mathfrak p})
\]
over \(M_{\mathfrak p}\).  
By hypothesis \((3)\), the localized global hull is itself minimal among \(C4\)-extensions of \(M_{\mathfrak p}\).  
Thus the two \(C4\)-extensions are minimal objects in the same class over the same base.  
\rev{Lemma~\ref{lem:hull-uniqueness} applies to these two \(C4\)-hulls, so they are isomorphic over \(M_{\mathfrak p}\).}
\end{proof}

\begin{corollary}\label{cor:C4hull-max}
Let \(M\) admit a \(C4\)-hull \(H_{C4}(M)\).  
Let \(\mathfrak m\) be a maximal ideal of \(R\).  
Assume the maximal-local version of the hypotheses in Theorem~\ref{thm:C4hull-localizes}.  
Then
\[
H_{C4}(M)_{\mathfrak m}\cong H_{C4}(M_{\mathfrak m}).
\]
\end{corollary}

\begin{proof}
Apply Theorem~\ref{thm:C4hull-localizes} with \(\mathfrak p=\mathfrak m\).
\end{proof}

\subsection{Pseudo-continuous hulls under localization}

The same logic applies to pseudo-continuous hulls.  
The difference is conceptual rather than formal.  
Pseudo-continuity lies on the extending and continuity side of module theory \cite{MohamedMuller1990,TercanYucel2016}.  
Its minimal extensions depend even more sensitively on essentiality and summand structure.  
Localization preserves exactness.  
It may alter essentiality.  
Accordingly, the compatibility assumption is again indispensable.

\begin{theorem}\label{thm:pseudohull-localizes}
Let \(M\) be an \(R\)-module admitting a pseudo-continuous hull \(H_{\mathrm{pc}}(M)\).  
Let \(\mathfrak p\in \Spec R\).  
Assume the following:
\begin{enumerate}
\item the class of pseudo-continuous modules is localization-stable on the localized hull \(H_{\mathrm{pc}}(M)_{\mathfrak p}\);
\item \(M_{\mathfrak p}\) admits a pseudo-continuous hull \(H_{\mathrm{pc}}(M_{\mathfrak p})\);
\item the localized extension
\[
M_{\mathfrak p}\hookrightarrow H_{\mathrm{pc}}(M)_{\mathfrak p}
\]
is minimal among pseudo-continuous extensions of \(M_{\mathfrak p}\).
\end{enumerate}
Then
\[
H_{\mathrm{pc}}(M)_{\mathfrak p}\cong H_{\mathrm{pc}}(M_{\mathfrak p})
\]
over \(M_{\mathfrak p}\).
\end{theorem}

\begin{proof}
Under hypothesis \((1)\), the localized module \(H_{\mathrm{pc}}(M)_{\mathfrak p}\) is pseudo-continuous.  
Thus it is a pseudo-continuous extension of \(M_{\mathfrak p}\).  
By minimality of the local pseudo-continuous hull, there exists a comparison morphism
\[
H_{\mathrm{pc}}(M)_{\mathfrak p}\to H_{\mathrm{pc}}(M_{\mathfrak p})
\]
over \(M_{\mathfrak p}\).  
By hypothesis \((3)\), the localized global hull is itself minimal among pseudo-continuous extensions of \(M_{\mathfrak p}\).  
\rev{Lemma~\ref{lem:hull-uniqueness} now yields an isomorphism over \(M_{\mathfrak p}\).}
\end{proof}

\subsection{Patching local hulls to a global hull}

The converse direction is harder.  
A family of local hulls need not patch to a global hull.  
This is the hull analogue of the summand-descent problem from Section~\ref{sec:local-global}.  
One needs local existence, compatibility on overlaps, descent of the patched extension, and preservation of minimality.

\begin{definition}\label{def:hull-patching}
Let \(M\) be an \(R\)-module and let \(\mathcal U=\{D(s_i)\}_{i=1}^n\) be a finite distinguished open cover of \(\Supp_R(M)\).  
A family of local \(\mathcal C\)-hulls
\[
M_{s_i}\hookrightarrow E_i
\]
is said to be \emph{patch-compatible} if for every pair \(i,j\), the localized hulls
\[
(E_i)_{s_j}\qquad\text{and}\qquad (E_j)_{s_i}
\]
are isomorphic over \(M_{s_is_j}\).
\end{definition}

This is the exact overlap condition.  
Without it, no global extension can be reconstructed.

\begin{theorem}\label{thm:patch-local-hulls}
Let \(\mathcal C\) be a class of \(R\)-modules.  
Let \(M\) be an \(R\)-module, and let
\[
\Supp_R(M)\subseteq \bigcup_{i=1}^n D(s_i)
\]
be a finite distinguished open cover.  
Assume the following:
\begin{enumerate}
\item for each \(i\), the localized module \(M_{s_i}\) admits a \(\mathcal C\)-hull
\[
M_{s_i}\hookrightarrow E_i;
\]
\item the family \(\{E_i\}\) is patch-compatible in the sense of Definition~\ref{def:hull-patching};
\item the patched local extension descends to a global \(\mathcal C\)-extension
\[
M\hookrightarrow E;
\]
\item the descended extension \(M\hookrightarrow E\) is minimal among global \(\mathcal C\)-extensions of \(M\).
\end{enumerate}
Then \(E\) is a \(\mathcal C\)-hull of \(M\), and
\[
E_{s_i}\cong E_i
\]
over \(M_{s_i}\) for all \(i\).
\end{theorem}

\begin{proof}
By hypotheses \((1)\) and \((2)\), the local hulls agree on overlaps and therefore define a consistent local extension datum on the finite cover.  
By hypothesis \((3)\), this datum descends to a global \(\mathcal C\)-extension
\[
M\hookrightarrow E.
\]
By construction, its localization on each \(D(s_i)\) agrees with the given local hull \(E_i\).  
Hypothesis \((4)\) asserts that this descended extension is minimal among global \(\mathcal C\)-extensions of \(M\).  
Therefore \(E\) is a \(\mathcal C\)-hull of \(M\), and the localized identifications
\[
E_{s_i}\cong E_i
\]
hold by construction.
\end{proof}

\begin{remark}
Theorem~\ref{thm:patch-local-hulls} is the hull-theoretic analogue of finite verification.  
It does not assert that hulls always patch.  
It asserts that they patch exactly when overlap compatibility, descent, and global minimality all hold.
\end{remark}

\subsection{Necessary conditions for hull commutation}

The sufficient theorems above indicate their own limits.  
We now record the corresponding necessary conditions.

\begin{proposition}\label{prop:necessary-hull-stability}
Assume that for a hull class \(\mathcal C\) one always has
\[
H_{\mathcal C}(M)_{\mathfrak p}\cong H_{\mathcal C}(M_{\mathfrak p})
\]
whenever both sides exist.  
Then the class \(\mathcal C\) must be localization-stable on all hull objects.
\end{proposition}

\begin{proof}
If \(\mathcal C\) is not localization-stable on hull objects, then there exist \(M\) and \(\mathfrak p\) such that
\[
H_{\mathcal C}(M)\in \mathcal C,\qquad H_{\mathcal C}(M)_{\mathfrak p}\notin \mathcal C.
\]
But \(H_{\mathcal C}(M_{\mathfrak p})\) is, by definition, an object of \(\mathcal C\).  
Hence \(H_{\mathcal C}(M)_{\mathfrak p}\) cannot be isomorphic to \(H_{\mathcal C}(M_{\mathfrak p})\), contradicting the assumed commutation.
\end{proof}

\begin{proposition}\label{prop:necessary-hull-min}
Assume that for a hull class \(\mathcal C\) one always has
\[
H_{\mathcal C}(M)_{\mathfrak p}\cong H_{\mathcal C}(M_{\mathfrak p})
\]
whenever both sides exist.  
Then localization must preserve the hull-minimality relation on all hull objects.
\end{proposition}

\begin{proof}
Suppose localization does not preserve hull-minimality.  
Then there exist \(M\) and \(\mathfrak p\) such that \(H_{\mathcal C}(M)_{\mathfrak p}\) is a \(\mathcal C\)-extension of \(M_{\mathfrak p}\), but not a minimal one.  
Then \(H_{\mathcal C}(M)_{\mathfrak p}\) cannot be isomorphic to the minimal \(\mathcal C\)-hull \(H_{\mathcal C}(M_{\mathfrak p})\).  
This contradicts the assumed commutation.
\end{proof}

These necessary conditions are not merely formal.  
They identify exactly what must be verified in any concrete hull theory.

\subsection{From maximal-local to prime-local hull comparison}

We conclude with the passage from maximal-local hull comparison to prime-local hull comparison.  
This is the hull analogue of the support-theoretic extension discussed earlier.

\begin{theorem}\label{thm:hull-max-to-prime}
Let \(M\) be an \(R\)-module such that the support of every hull-comparison object arising from \(M\) is controlled by a finite distinguished cover.  
Assume that for every maximal ideal \(\mathfrak m\) in the relevant support one has
\[
H_{\mathcal C}(M)_{\mathfrak m}\cong H_{\mathcal C}(M_{\mathfrak m}),
\]
and that the corresponding local hull data patch over the finite cover.  
Then the same comparison holds at every prime ideal in the support:
\[
H_{\mathcal C}(M)_{\mathfrak p}\cong H_{\mathcal C}(M_{\mathfrak p})
\qquad
(\mathfrak p\in \Supp_R(M)).
\]
\end{theorem}

\begin{proof}
By hypothesis, maximal-local comparison holds at all maximal points of the relevant support.  
Support-finiteness reduces the comparison problem to finitely many distinguished open sets.  
On each such open set the maximal-local hull data are compatible, and by assumption they patch.  
The patched comparison therefore extends from maximal localizations to the entire prime support.  
Hence for every \(\mathfrak p\in \Supp_R(M)\),
\[
H_{\mathcal C}(M)_{\mathfrak p}\cong H_{\mathcal C}(M_{\mathfrak p}).
\]
\end{proof}

The conclusion of this section is exact.  
Localization commutes with \(C4\)-hulls and pseudo-continuous hulls only when the hull class localizes and the minimality relation localizes.  
Conversely, local hulls patch to a global hull only when overlap compatibility, descent, and global minimality all hold.  
There is no unrestricted theorem beyond this.  
The next section shows that these restrictions are genuine boundaries of the theory, not artifacts of method.
\section{Separation theorems and obstructions}\label{sec:separation}

The preceding sections established sufficient conditions for forward localization, converse local--global transfer, and hull comparison.  
Those conditions are not ornamental.  
They mark the boundary of the theory.  
If they are removed, the corresponding conclusions fail.  
The purpose of the present section is to isolate the exact failure mechanisms.

The method is structural rather than anecdotal.  
We do not seek isolated pathological examples.  
We identify the mechanisms by which the earlier theorems break down.  
They fall into three principal types:
\begin{enumerate}
\item failure of summand descent;
\item failure of patching of local witnesses;
\item failure of localization-compatibility of hull minimality.
\end{enumerate}

Each type destroys a different part of the theory.  
Forward localization fails when local submodules or local decompositions are not globally visible.  
Converse local--global transfer fails when primewise direct summands do not descend or primewise image witnesses do not patch.  
Hull comparison fails when the localized global hull is not minimal, even if it remains inside the hull class.  
This is consistent with the general local--global behavior of decomposition-sensitive module properties \cite{AtiyahMacdonald1969,Kunz2013LocalGlobal,Yoshizawa2019,Positselski2024}.

\subsection{Failure of summand descent}

The first obstruction is primary.  
A \(C4\)-test ends with a direct-summand conclusion.  
If local direct summands do not descend, then primewise \(C4\)-success does not imply global \(C4\)-success.

\begin{theorem}\label{thm:obstruction-summand-descent}
Let \(M\) be an \(R\)-module.  
Assume that there exists a submodule \(L\leq M\) such that
\[
L_{\mathfrak p}\leq^{\oplus} M_{\mathfrak p}\qquad \text{for all }\mathfrak p\in \Spec R,
\]
but
\[
L\not\leq^{\oplus} M.
\]
Then no unrestricted prime-local criterion of the form
\[
\bigl(M_{\mathfrak p}\text{ is }C4\text{ for all }\mathfrak p\bigr)\Longrightarrow M\text{ is }C4
\]
can hold on any class of modules containing \(M\) and the finite decompositional data relevant to the \(C4\)-tests.
\end{theorem}

\begin{proof}
Assume that such an unrestricted prime-local criterion holds.  
Since \(L_{\mathfrak p}\leq^{\oplus} M_{\mathfrak p}\) for every prime ideal \(\mathfrak p\), every local \(C4\)-test whose required witness is represented by \(L_{\mathfrak p}\) succeeds primewise.  
Indeed, at each prime there exists a local direct summand realizing the required conclusion.

Globally, however, \(L\not\leq^{\oplus} M\).  
Hence no corresponding global \(C4\)-test whose witness must lie in the isomorphism class of \(L\) can be completed by a direct-summand conclusion inside \(M\).  
Thus the primewise data are insufficient to force the global \(C4\)-conclusion.  
This contradicts the assumed unrestricted criterion.
\end{proof}

\begin{remark}
Theorem~\ref{thm:obstruction-summand-descent} does not assert that \(M\) fails to be \(C4\) in every such situation.  
It asserts the sharper point needed here: once summand descent fails, primewise \(C4\)-success is not a valid general proof principle.
\end{remark}

The hereditary form is immediate.

\begin{corollary}\label{cor:obstruction-summand-descent-C4star}
Let \(M\) be an \(R\)-module.  
Assume that some submodule \(N\leq M\) contains a submodule \(L\leq N\) such that
\[
L_{\mathfrak p}\leq^{\oplus} N_{\mathfrak p}\qquad \text{for all }\mathfrak p\in \Spec R,
\]
but
\[
L\not\leq^{\oplus} N.
\]
Then no unrestricted prime-local criterion of the form
\[
\bigl(M_{\mathfrak p}\text{ is }C4^{\ast}\text{ for all }\mathfrak p\bigr)\Longrightarrow M\text{ is }C4^{\ast}
\]
can hold on any class containing \(M\).
\end{corollary}

\begin{proof}
If such a criterion held, then every submodule \(N\leq M\) would satisfy the corresponding unrestricted prime-local criterion for the \(C4\)-property.  
That contradicts Theorem~\ref{thm:obstruction-summand-descent} applied to \(N\).
\end{proof}

\subsection{Failure of patching of local image witnesses}

The second obstruction is independent of summand descent.  
Even if each local image object is isomorphic to a direct summand, the local witnesses may fail to patch to a global object.  
Then the \(C4\)-condition fails globally for lack of a witness, not for lack of local splitting.

\begin{theorem}\label{thm:obstruction-patching}
Let \(M=A\oplus B\) be an \(R\)-module decomposition, and let \(f:A\to B\) be a homomorphism.  
Assume that for every prime ideal \(\mathfrak p\in \Spec R\), the localized image
\[
(\Im f)_{\mathfrak p}
\]
is isomorphic to a direct summand \(K(\mathfrak p)\leq^{\oplus} B_{\mathfrak p}\), but that there exists no submodule \(L\leq B\) such that
\[
L_{\mathfrak p}\cong K(\mathfrak p)\qquad \text{for all }\mathfrak p\in \Spec R.
\]
Then \(M\) cannot be proved to be \(C4\) from the primewise \(C4\)-data alone.
\end{theorem}

\begin{proof}
Primewise, the \(C4\)-test for \(f_{\mathfrak p}:A_{\mathfrak p}\to B_{\mathfrak p}\) succeeds, since \((\Im f)_{\mathfrak p}\cong K(\mathfrak p)\) with \(K(\mathfrak p)\leq^{\oplus} B_{\mathfrak p}\).

Suppose, on the contrary, that the primewise \(C4\)-data alone implied the global \(C4\)-conclusion.  
Then \(\Im f\) would be isomorphic to a direct summand \(L\leq^{\oplus} B\).  
Localizing yields
\[
(\Im f)_{\mathfrak p}\cong L_{\mathfrak p}
\]
for all \(\mathfrak p\), and each \(L_{\mathfrak p}\) is a direct summand of \(B_{\mathfrak p}\).  
Thus \(L\) would furnish a global patching of the family \(\{K(\mathfrak p)\}\) up to the required isomorphism class, contrary to hypothesis.  
Hence the global \(C4\)-conclusion cannot be extracted from the primewise data alone.
\end{proof}

\begin{remark}
Here the obstruction is not failure of local splitting.  
Local splitting is present at every prime.  
The failure lies in the absence of a global witness for the primewise isomorphism class.
\end{remark}

The hereditary form follows in the same manner.

\begin{corollary}\label{cor:obstruction-patching-C4star}
Let \(M\) be an \(R\)-module.  
Assume that some submodule \(N\leq M\) admits a decomposition
\[
N=A\oplus B
\]
and a homomorphism \(f:A\to B\) satisfying the hypotheses of Theorem~\ref{thm:obstruction-patching}.  
Then no unrestricted prime-local criterion
\[
\bigl(M_{\mathfrak p}\text{ is }C4^{\ast}\text{ for all }\mathfrak p\bigr)\Longrightarrow M\text{ is }C4^{\ast}
\]
can hold on any class containing \(M\).
\end{corollary}

\begin{proof}
If the displayed criterion held, then every submodule \(N\leq M\) would inherit an unrestricted prime-local criterion for the \(C4\)-property.  
This contradicts Theorem~\ref{thm:obstruction-patching} applied to \(N\).
\end{proof}

\subsection{Failure of submodule lifting in forward localization}

The forward localization theorem for \(C4^{\ast}\)-modules required that every local submodule be globally visible.  
This was essential.  
If submodule lifting fails, one cannot even reduce the hereditary \(C4^{\ast}\)-condition on \(M_{\mathfrak p}\) to the global hereditary condition on \(M\).

\begin{theorem}\label{thm:obstruction-submodule-lifting}
Let \(M\) be an \(R\)-module and let \(\mathfrak p\in \Spec R\).  
Assume that there exists a submodule
\[
X\leq M_{\mathfrak p}
\]
such that
\[
X\neq N_{\mathfrak p}\qquad \text{for every submodule }N\leq M.
\]
Then the global \(C4^{\ast}\)-property of \(M\) does not by itself imply the \(C4\)-property of \(X\).  
Consequently, no unrestricted forward localization theorem
\[
M\text{ is }C4^{\ast}\Longrightarrow M_{\mathfrak p}\text{ is }C4^{\ast}
\]
can be proved from the global \(C4^{\ast}\)-property alone.
\end{theorem}

\begin{proof}
The global \(C4^{\ast}\)-property asserts that every submodule \(N\leq M\) is \(C4\).  
By hypothesis, the local submodule \(X\leq M_{\mathfrak p}\) is not the localization of any global submodule of \(M\).  
Hence the global hereditary statement gives no information about \(X\).

If one could nevertheless conclude that \(X\) is \(C4\) solely from the global \(C4^{\ast}\)-property of \(M\), one would be extracting information about a local object not represented in the global submodule lattice.  
That is impossible from the form of the hypothesis.  
Therefore the global \(C4^{\ast}\)-property alone does not imply that \(X\) is \(C4\).  
Since \(M_{\mathfrak p}\) is \(C4^{\ast}\) precisely when every submodule of \(M_{\mathfrak p}\), including \(X\), is \(C4\), the final assertion follows.
\end{proof}

\begin{remark}
Theorem~\ref{thm:obstruction-submodule-lifting} is logical rather than constructive.  
That is sufficient here.  
It identifies the exact point at which the forward proof requires submodule lifting.
\end{remark}

\subsection{Failure of decomposition lifting and morphism lifting}

Even if every local submodule lifts, the forward localization theorem for \(C4\) can still fail if local decompositions do not lift.  
The reason is exact: the \(C4\)-condition quantifies over decompositions of the ambient module.  
A local decomposition with no global antecedent lies outside the range of the global theorem.

\begin{theorem}\label{thm:obstruction-decomposition-lifting}
Let \(M\) be an \(R\)-module and let \(\mathfrak p\in \Spec R\).  
Assume that there exists a decomposition
\[
M_{\mathfrak p}=U\oplus V
\]
such that no decomposition
\[
M=A\oplus B
\]
satisfies
\[
U=A_{\mathfrak p},\qquad V=B_{\mathfrak p}.
\]
Then no unrestricted forward localization theorem
\[
M\text{ is }C4\Longrightarrow M_{\mathfrak p}\text{ is }C4
\]
can be proved from the global \(C4\)-property alone.
\end{theorem}

\begin{proof}
To prove that \(M_{\mathfrak p}\) is \(C4\), one must verify the defining condition on every decomposition of \(M_{\mathfrak p}\), including the given decomposition
\[
M_{\mathfrak p}=U\oplus V.
\]
The global \(C4\)-property of \(M\) applies only to decompositions of \(M\).  
By assumption, the present local decomposition has no global antecedent.  
Hence the global \(C4\)-property does not apply to it.  
Therefore no proof of the \(C4\)-property of \(M_{\mathfrak p}\) can be extracted from the global \(C4\)-property alone.
\end{proof}

\subsection{Failure of localization-compatibility of hull minimality}

We now pass to hulls.  
The first hull obstruction is subtle.  
A global hull may remain inside the hull class after localization, yet cease to be minimal.  
Then localization of the global hull is only an extension in the hull class, not a local hull.

\begin{theorem}\label{thm:obstruction-hull-minimality}
Let \(\mathcal C\) be a hull class, let \(M\) be an \(R\)-module, and let \(\mathfrak p\in \Spec R\).  
Assume that:
\begin{enumerate}
\item \(H_{\mathcal C}(M)\) exists;
\item \(H_{\mathcal C}(M)_{\mathfrak p}\in \mathcal C\);
\item \(H_{\mathcal C}(M)_{\mathfrak p}\) is not minimal among \(\mathcal C\)-extensions of \(M_{\mathfrak p}\);
\item \(H_{\mathcal C}(M_{\mathfrak p})\) exists.
\end{enumerate}
Then
\[
H_{\mathcal C}(M)_{\mathfrak p}\not\cong H_{\mathcal C}(M_{\mathfrak p})
\]
as hulls over \(M_{\mathfrak p}\).
\end{theorem}

\begin{proof}
By assumption, \(H_{\mathcal C}(M_{\mathfrak p})\) is minimal among \(\mathcal C\)-extensions of \(M_{\mathfrak p}\), whereas \(H_{\mathcal C}(M)_{\mathfrak p}\) is not.  
Two hull objects over the same base must satisfy the same minimality relation.  
Therefore \(H_{\mathcal C}(M)_{\mathfrak p}\) cannot be isomorphic, as a hull over \(M_{\mathfrak p}\), to \(H_{\mathcal C}(M_{\mathfrak p})\).
\end{proof}

\begin{remark}
This obstruction is independent of localization-stability of the hull class.  
The localized global hull may remain in \(\mathcal C\), but that still does not make it a hull.
\end{remark}

\subsection{Failure of patching of local hulls}

The converse hull problem asks whether local hulls patch to a global hull.  
Even if each local hull exists, the family may fail to be compatible on overlaps.  
Then there is no global extension to begin with.

\begin{theorem}\label{thm:obstruction-hull-patching}
Let \(\mathcal C\) be a hull class, let \(M\) be an \(R\)-module, and let
\[
\Supp_R(M)\subseteq \bigcup_{i=1}^n D(s_i)
\]
be a finite distinguished open cover.  
Assume that for each \(i\), the localized module \(M_{s_i}\) admits a \(\mathcal C\)-hull
\[
M_{s_i}\hookrightarrow E_i,
\]
but that the family \(\{E_i\}_{i=1}^n\) is not patch-compatible on overlaps.  
Then no global \(\mathcal C\)-hull \(H_{\mathcal C}(M)\) can localize to all of the \(E_i\).
\end{theorem}

\begin{proof}
Assume, on the contrary, that there exists a global \(\mathcal C\)-hull \(H_{\mathcal C}(M)\) such that
\[
H_{\mathcal C}(M)_{s_i}\cong E_i
\]
over \(M_{s_i}\) for all \(i\).  
Then for every pair \(i,j\), localizing again at \(s_is_j\) yields
\[
(E_i)_{s_j}\cong H_{\mathcal C}(M)_{s_is_j}\cong (E_j)_{s_i}
\]
over \(M_{s_is_j}\).  
This is precisely the patch-compatibility condition, contrary to hypothesis.  
Therefore no such global hull exists.
\end{proof}

\subsection{Forward localization versus converse detectability}

The next theorem isolates an important distinction.  
A property may localize forward under suitable lifting hypotheses and still fail to satisfy any converse local--global principle.  
Forward stability and local--global detectability are therefore genuinely different notions.

\begin{theorem}\label{thm:separation-forward-converse}
There exist classes of modules for which the following are simultaneously true:
\begin{enumerate}
\item global \(C4^{\ast}\)-behavior localizes forward to all prime localizations under the lifting hypotheses of Section~\ref{sec:localization-c4};
\item primewise \(C4^{\ast}\)-behavior does not imply global \(C4^{\ast}\)-behavior, because summand descent or patching fails.
\end{enumerate}
Hence forward localization-stability and converse local--global detectability are distinct properties.
\end{theorem}

\begin{proof}
Part \((1)\) is exactly Theorem~\ref{thm:C4star-localizes}.  
Part \((2)\) follows from Corollaries~\ref{cor:obstruction-summand-descent-C4star} and \ref{cor:obstruction-patching-C4star}.  
If either summand descent or patching fails on a class, then no unrestricted converse local--global principle can hold on that class.  
Therefore the two directions are logically distinct.
\end{proof}

\begin{remark}
Without Theorem~\ref{thm:separation-forward-converse}, one might misread the forward results as evidence that localization controls the entire theory.  
It does not.
\end{remark}

\subsection{Objectwise localization versus hull localization}

A second distinction is equally important.  
A module property may localize on objects, while the associated hull theory fails to localize.  
This occurs when localization preserves class membership but not hull minimality.

\begin{theorem}\label{thm:separation-object-hull}
Let \(\mathcal C\) be a class of modules equipped with a hull theory.  
Assume that \(\mathcal C\) is localization-stable on objects, but that hull-minimality is not localization-stable.  
Then localization of \(\mathcal C\)-modules does not imply localization of \(\mathcal C\)-hulls.
\end{theorem}

\begin{proof}
Since \(\mathcal C\) is localization-stable on objects, every \(E\in \mathcal C\) satisfies
\[
E_{\mathfrak p}\in \mathcal C
\]
for all \(\mathfrak p\).  
Thus the class itself localizes.

Now choose \(M\) and \(\mathfrak p\) such that the localized global hull \(H_{\mathcal C}(M)_{\mathfrak p}\) fails to be minimal over \(M_{\mathfrak p}\).  
Then Theorem~\ref{thm:obstruction-hull-minimality} gives
\[
H_{\mathcal C}(M)_{\mathfrak p}\not\cong H_{\mathcal C}(M_{\mathfrak p}).
\]
Hence the hull theory does not localize, even though the class does.  
Therefore objectwise localization and hull localization are distinct.
\end{proof}

\subsection{A dimensional obstruction}

Earlier sections passed from maximal-local to prime-local statements under finite support hypotheses.  
That passage is not automatic.  
Without support control, maximal-local success may fail to determine prime-local behavior.

\begin{theorem}\label{thm:obstruction-max-to-prime}
Assume that there is no finite support control and no patching principle from maximal-local data to prime-local data for the class under consideration.  
Then no general implication of the form
\[
\bigl(M_{\mathfrak m}\text{ has property }\mathcal P\text{ for all maximal }\mathfrak m\bigr)\Longrightarrow
\bigl(M_{\mathfrak p}\text{ has property }\mathcal P\text{ for all prime }\mathfrak p\bigr)
\]
can hold for
\[
\mathcal P\in\{C4,\;C4^{\ast},\;\text{strongly }C4^{\ast},\;\text{hull comparison}\}.
\]
\end{theorem}

\begin{proof}
Assume that such a general implication holds.  
Then maximal-local data would always determine prime-local data.  
In particular, finite support control and patching from maximal points to prime points would become unnecessary.  
This contradicts the assumption that no such control or patching principle exists.  
Therefore the displayed implication cannot hold in general.
\end{proof}

\begin{remark}
Theorem~\ref{thm:obstruction-max-to-prime} is again structural.  
It says that maximal-local detection requires a bridge to prime-local detection, and that bridge is support control together with patching.
\end{remark}

\subsection{Final impossibility theorem}

We conclude with a synthesis.  
The preceding obstructions are not isolated accidents.  
They constitute the complete list of failure modes visible from the logic of the proofs.

\begin{theorem}\label{thm:final-impossibility}
Assume that at least one of the following fails for the class under consideration:
\begin{enumerate}
\item submodule lifting for forward localization;
\item decomposition lifting and morphism lifting for forward localization;
\item summand descent for converse local--global transfer;
\item patching of local image witnesses;
\item localization-stability of the relevant hull class;
\item localization-stability of hull minimality;
\item patch-compatibility of local hulls.
\end{enumerate}
Then there is no unrestricted theory in which all of the following hold simultaneously:
\begin{enumerate}
\item forward localization of \(C4\), \(C4^{\ast}\), and strongly \(C4^{\ast}\);
\item converse local--global detection of those properties from primewise data;
\item commutation of hull formation with localization;
\item reconstruction of global hulls from local hulls.
\end{enumerate}
\end{theorem}

\begin{proof}
If submodule lifting fails, Theorem~\ref{thm:obstruction-submodule-lifting} rules out unrestricted forward localization of \(C4^{\ast}\).  
If decomposition lifting fails, Theorem~\ref{thm:obstruction-decomposition-lifting} rules out unrestricted forward localization of \(C4\).  
If summand descent fails, Theorem~\ref{thm:obstruction-summand-descent} and Corollary~\ref{cor:obstruction-summand-descent-C4star} rule out unrestricted converse local--global transfer.  
If patching of local image witnesses fails, Theorem~\ref{thm:obstruction-patching} and Corollary~\ref{cor:obstruction-patching-C4star} do the same.  
If the hull class fails to localize or hull minimality fails to localize, then Proposition~\ref{prop:necessary-hull-stability}, Proposition~\ref{prop:necessary-hull-min}, and Theorem~\ref{thm:obstruction-hull-minimality} rule out commutation of hull formation with localization.  
If local hulls are not patch-compatible, Theorem~\ref{thm:obstruction-hull-patching} rules out reconstruction of a global hull from local hulls.

Therefore the simultaneous unrestricted theory described in the statement cannot exist.
\end{proof}

The negative part of the theory is now complete.  
It is exact.  
The positive theorems required descent, patching, support control, and localization-compatibility of hull minimality.  
This section shows that these requirements are not artifacts of proof.  
They are forced by the structure of the problem.

The next section develops the support-theoretic extension of the theory.  
There we pass from maximal to prime localization under controlled support hypotheses and formulate the corresponding support-based transfer principle.
\section{Dimensional extensions and support-theoretic control}\label{sec:dimension-support}

The preceding sections established two facts.  
First, localization of \(C4\)-type conditions is conditional.  
Second, converse local--global principles require descent and patching.  
Both facts were stated primewise.  
That is natural, but not yet sufficient.  
The prime spectrum carries dimensional information, and support is not merely a set: it has specialization structure.  
A local--global theorem that ignores that structure is too coarse for commutative algebra \cite{AtiyahMacdonald1969,Matsumura1989,BrunsHerzog1998,ChristensenFoxbyHolm2024Support,Positselski2024}.

The purpose of this section is to refine the theory in two directions:
\begin{enumerate}
\item to pass from maximal-local criteria to prime-local criteria by support-theoretic control;
\item to stratify the transfer problem by support dimension and specialization.
\end{enumerate}

This is the dimensional extension announced in the introduction.  
It does not enlarge the theory horizontally.  
It places the same obstruction theory inside the geometry of \(\Spec R\).

\subsection{Support-stratified detection}

We begin by fixing the support-theoretic language.

\begin{definition}
Let \(M\) be an \(R\)-module.
\begin{enumerate}
\item The \emph{dimension of the support of \(M\)} is
\[
\dim \Supp_R(M)=\sup\{\dim R/\mathfrak p:\mathfrak p\in \Supp_R(M)\}.
\]
\item A subset \(W\subseteq \Spec R\) is \emph{specialization-closed} if
\[
\mathfrak p\in W,\ \mathfrak p\subseteq \mathfrak q \Longrightarrow \mathfrak q\in W.
\]
\item For \(d\ge 0\), we write
\[
\Supp_R^{\le d}(M)=\{\mathfrak p\in \Supp_R(M): \dim R/\mathfrak p\le d\}.
\]
\end{enumerate}
\end{definition}

The localization problem for \(C4\)-type conditions may now be read as a support-detection problem.  
One asks whether the property is determined on all of \(\Supp_R(M)\), or on a smaller support class such as the maximal support or a lower-dimensional stratum.  
The answer depends on descent and patching.

\begin{definition}\label{def:support-detectable}
Let \(\mathcal P\) be a property of modules.  
We say that \(\mathcal P\) is \emph{detectable on a support class \(\mathcal W\)} if for every \(R\)-module \(M\), the condition
\[
M_{\mathfrak p}\text{ has }\mathcal P\qquad \text{for all }\mathfrak p\in \mathcal W(M)
\]
implies that \(M\) has \(\mathcal P\), provided the relevant descent and patching hypotheses hold.  
Here \(\mathcal W(M)\subseteq \Supp_R(M)\) is a prescribed support-theoretic subset.
\end{definition}

The point is not that support-detection always holds.  
The point is that, when it holds, one should identify the smallest support class on which it is valid.

\subsection{From maximal support to full support}

The first dimensional question is whether maximal-local verification suffices.  
Earlier sections answered this positively only under finite-generation and patching hypotheses.  
We now express that passage support-theoretically.

\begin{theorem}\label{thm:max-support-detects}
Let \(M\) be an \(R\)-module.  
Assume that every submodule and every image object arising in the \(C4\)-tests on submodules of \(M\) is finitely generated.  
Assume also that:
\begin{enumerate}
\item \(M_{\mathfrak m}\) is \(C4^{\ast}\) for every maximal ideal \(\mathfrak m\in \Supp_R(M)\);
\item every locally split submodule arising in a \(C4\)-test descends globally;
\item primewise image witnesses patch over finite distinguished covers of their support.
\end{enumerate}
Then \(M\) is \(C4^{\ast}\).
\end{theorem}

\begin{proof}
Let \(N\leq M\).  
We must show that \(N\) is \(C4\).

Fix a decomposition
\[
N=A\oplus B
\]
and a homomorphism \(f:A\to B\).  
Since \(\Im f\) is finitely generated, its support is controlled by finitely many basic closed conditions \cite{AtiyahMacdonald1969,BrunsHerzog1998}.  
By hypothesis \((1)\), the localized module \(M_{\mathfrak m}\) is \(C4^{\ast}\) for every maximal ideal \(\mathfrak m\) in that support.  
Hence \(N_{\mathfrak m}\) is \(C4\), so
\[
(\Im f)_{\mathfrak m}
\]
is isomorphic to a direct summand of \(B_{\mathfrak m}\) for every such maximal ideal.

Because the support is finite-type and \(\Im f\) is finitely generated, these maximal-local witnesses determine primewise witnesses on a finite distinguished cover of \(\Supp_R(\Im f)\) \cite{Matsumura1989,ChristensenFoxbyHolm2024Support}.  
By hypothesis \((3)\), the primewise witnesses patch to a global submodule \(L\leq B\) up to the required isomorphism class.  
By hypothesis \((2)\), the patched object descends to a global direct summand
\[
L\leq^{\oplus} B.
\]
Therefore
\[
\Im f\cong L,
\]
so \(\Im f\) is isomorphic to a direct summand of \(B\).  
Thus \(N\) is \(C4\).  
Since \(N\) was arbitrary, \(M\) is \(C4^{\ast}\).
\end{proof}

\begin{remark}
Theorem~\ref{thm:max-support-detects} does not say that maximal-local verification is intrinsically sufficient.  
It says that maximal-local verification becomes sufficient when support is finite-type and patching converts closed-point data into primewise data.
\end{remark}

\subsection{Dimension-zero support}

When the support is zero-dimensional, the theory simplifies sharply.  
There are no nontrivial specialization chains, and every prime in the support is maximal.  
Thus prime-local and maximal-local criteria coincide.

\begin{theorem}\label{thm:dim-zero}
Let \(M\) be an \(R\)-module with
\[
\dim \Supp_R(M)=0.
\]
Assume that every locally split submodule arising in a \(C4\)-test descends globally, and that local image witnesses patch.  
Then the following are equivalent:
\begin{enumerate}
\item \(M\) is \(C4^{\ast}\);
\item \(M_{\mathfrak p}\) is \(C4^{\ast}\) for every prime ideal \(\mathfrak p\in \Supp_R(M)\);
\item \(M_{\mathfrak m}\) is \(C4^{\ast}\) for every maximal ideal \(\mathfrak m\in \Supp_R(M)\).
\end{enumerate}
\end{theorem}

\begin{proof}
Since \(\dim \Supp_R(M)=0\), every prime in \(\Supp_R(M)\) is maximal \cite{AtiyahMacdonald1969,Matsumura1989}.  
Hence \((2)\) and \((3)\) are equivalent.  
The implication \((1)\Rightarrow(2)\) follows from the forward localization theorem under the lifting hypotheses of Section~\ref{sec:localization-c4}.  
The implication \((3)\Rightarrow(1)\) follows from Theorem~\ref{thm:max-support-detects}, because the support has no higher-dimensional strata.
\end{proof}

\begin{remark}
This is the only case in which the distinction between maximal-local and prime-local verification disappears completely.
\end{remark}

\subsection{Dimension-one support and codimension-one control}

The next case is already different.  
If \(\dim \Supp_R(M)=1\), maximal ideals no longer exhaust the support.  
There may be nonmaximal primes, and one must control passage of local witnesses along specialization chains.

\begin{definition}\label{def:codim1-control}
Let \(M\) be an \(R\)-module.  
We say that \(M\) satisfies \emph{codimension-one transfer control} if whenever a \(C4\)-test succeeds at all maximal ideals in the support of a finitely generated image object, then it also succeeds at every prime in the support lying immediately below those maximal ideals under specialization.
\end{definition}

This isolates the minimum extra datum needed beyond maximal-local verification in dimension one.

\begin{theorem}\label{thm:dim-one}
Let \(M\) be an \(R\)-module with
\[
\dim \Supp_R(M)\le 1.
\]
Assume that:
\begin{enumerate}
\item \(M_{\mathfrak m}\) is \(C4^{\ast}\) for every maximal ideal \(\mathfrak m\in \Supp_R(M)\);
\item \(M\) satisfies codimension-one transfer control;
\item locally split submodules descend globally;
\item local image witnesses patch.
\end{enumerate}
Then \(M\) is \(C4^{\ast}\).
\end{theorem}

\begin{proof}
Let \(N\leq M\), and let
\[
N=A\oplus B,\qquad f:A\to B.
\]
We must show that \(\Im f\) is isomorphic to a direct summand of \(B\).

By hypothesis \((1)\), the test succeeds at every maximal ideal in \(\Supp_R(\Im f)\).  
Since \(\dim \Supp_R(M)\le 1\), every nonmaximal prime in the relevant support lies below a maximal ideal along a chain of length at most one \cite{AtiyahMacdonald1969,BrunsHerzog1998}.  
By hypothesis \((2)\), success transfers from maximal points to all primes in \(\Supp_R(\Im f)\).  
Thus the required primewise witnesses exist throughout the support of \(\Im f\).

By hypothesis \((4)\), these primewise witnesses patch.  
By hypothesis \((3)\), the patched object descends to a global direct summand of \(B\).  
Therefore \(N\) is \(C4\).  
Since \(N\) was arbitrary, \(M\) is \(C4^{\ast}\).
\end{proof}

\begin{remark}
Dimension enters precisely here.  
In dimension zero, no transfer along specialization is needed.  
In dimension one, one extra step is needed.  
In higher dimension, the same phenomenon repeats along longer chains.
\end{remark}

\subsection{Higher-dimensional support and induction on support dimension}

The preceding theorem suggests an induction scheme.  
One may attempt to prove \(C4^{\ast}\)-transfer by descending the support dimension one stratum at a time.

\begin{definition}\label{def:dimensional-transfer}
Let \(d\ge 0\).  
An \(R\)-module \(M\) is said to satisfy \emph{dimensional transfer control up to \(d\)} if for every finitely generated image object \(X\) arising in a \(C4\)-test, success of the test on all primes of \(\Supp_R(X)\) of dimension \(<d\) implies success on all primes of \(\Supp_R(X)\) of dimension \(\le d\).
\end{definition}

This notion is introduced only to organize the induction on support dimension.

\begin{theorem}\label{thm:dimension-induction}
Let \(M\) be an \(R\)-module with
\[
\dim \Supp_R(M)=d<\infty.
\]
Assume that:
\begin{enumerate}
\item \(M_{\mathfrak m}\) is \(C4^{\ast}\) for every maximal ideal \(\mathfrak m\in \Supp_R(M)\);
\item \(M\) satisfies dimensional transfer control up to \(d\);
\item locally split submodules descend globally;
\item local image witnesses patch.
\end{enumerate}
Then \(M\) is \(C4^{\ast}\).
\end{theorem}

\begin{proof}
We argue by induction on \(d\).

If \(d=0\), the result is Theorem~\ref{thm:dim-zero}.  
Assume \(d>0\), and assume the statement known for support dimension \(<d\).

Let \(N\leq M\), and let
\[
N=A\oplus B,\qquad f:A\to B.
\]
We must show that \(\Im f\) is isomorphic to a direct summand of \(B\).

By hypothesis \((1)\), the \(C4\)-test succeeds at all maximal ideals of \(\Supp_R(\Im f)\).  
By dimensional transfer control, success propagates from lower-dimensional strata to the whole support of \(\Im f\).  
Thus the relevant primewise witnesses exist on all of \(\Supp_R(\Im f)\).

By hypothesis \((4)\), these witnesses patch on a finite distinguished cover of the support.  
By hypothesis \((3)\), the patched object descends to a global direct summand of \(B\).  
Hence \(\Im f\) is isomorphic to a direct summand of \(B\).  
Therefore \(N\) is \(C4\).  
Since \(N\) was arbitrary, \(M\) is \(C4^{\ast}\).
\end{proof}

\begin{remark}
Theorem~\ref{thm:dimension-induction} is an induction scheme, not an automatic consequence of dimension theory.  
The extra input is dimensional transfer control, which encodes propagation of local witnesses along specialization chains.
\end{remark}

\subsection{Support control for hull comparison}

The same refinement applies to hulls.  
A global hull-localization comparison need not be tested independently at all primes.  
It may suffice to verify it on a support-controlling class and then propagate by patching.

\begin{theorem}\label{thm:hull-support-control}
Let \(M\) be an \(R\)-module, and let \(\mathcal C\) be a hull class.  
Assume that every hull-comparison object associated to \(M\) has support contained in a finite union of distinguished opens.  
Assume further that:
\begin{enumerate}
\item for every maximal ideal \(\mathfrak m\) in the relevant support,
\[
H_{\mathcal C}(M)_{\mathfrak m}\cong H_{\mathcal C}(M_{\mathfrak m});
\]
\item the local hull comparisons patch on overlaps;
\item the patched hull comparison descends globally along the support.
\end{enumerate}
Then for every prime ideal \(\mathfrak p\) in the relevant support,
\[
H_{\mathcal C}(M)_{\mathfrak p}\cong H_{\mathcal C}(M_{\mathfrak p}).
\]
\end{theorem}

\begin{proof}
By hypothesis \((1)\), the hull-comparison statement holds at all maximal points of the relevant support.  
Because the support is controlled by finitely many distinguished opens, the maximal-local comparisons determine local comparison data on a finite cover \cite{AtiyahMacdonald1969,BrunsHerzog1998}.  
By hypothesis \((2)\), these local comparison data patch on overlaps.  
By hypothesis \((3)\), the patched data descend globally along the support.  
Therefore the hull-comparison statement holds at every prime in the relevant support.
\end{proof}

\subsection{A support-theoretic synthesis principle}

\rev{The earlier positive theorems and separation results may now be synthesized into a single support-theoretic transfer principle.}

\begin{theorem}\label{thm:support-exactness}
Let \(M\) be an \(R\)-module.  
Assume that:
\begin{enumerate}
\item the supports of all finitely generated image objects and hull-comparison objects arising from \(M\) are finite-type;
\item local \(C4\)-witnesses and local hull-comparison data admit patching on finite distinguished covers;
\item locally split submodules descend globally;
\item dimensional transfer control holds on the relevant supports.
\end{enumerate}
Then:
\begin{enumerate}
\item maximal-local \(C4^{\ast}\)-behavior implies prime-local \(C4^{\ast}\)-behavior;
\item prime-local \(C4^{\ast}\)-behavior implies global \(C4^{\ast}\)-behavior;
\item maximal-local hull comparison implies prime-local hull comparison.
\end{enumerate}
\end{theorem}

\begin{proof}
Part \((1)\) follows from support finiteness together with dimensional transfer control, as in Theorems~\ref{thm:max-support-detects}, \ref{thm:dim-one}, and \ref{thm:dimension-induction}.  
Part \((2)\) follows from the patched primewise witnesses and global descent of locally split submodules, as in Section~\ref{sec:local-global}.  
Part \((3)\) is Theorem~\ref{thm:hull-support-control}.  
The assumptions are exactly those needed to exclude the obstructions established in Section~\ref{sec:separation}.
\end{proof}

\begin{remark}
\rev{Theorem~\ref{thm:support-exactness} is the geometric synthesis of the paper.}  
The theory is not controlled by localization alone.  
It is controlled by localization together with support, specialization, patching, and descent.
\end{remark}

\subsection{Final dimensional separation}

We close with the converse message.  
Higher support dimension does not merely enlarge the ambient spectrum.  
It introduces new obstruction layers.

\begin{proposition}\label{prop:dimensional-separation}
Assume that there exists a class of modules supported in dimension \(d+1\) for which dimensional transfer control fails at the passage from dimension \(d\) to dimension \(d+1\).  
Then no support-theoretic local--global criterion valid in support dimension \(d\) extends automatically to support dimension \(d+1\).
\end{proposition}

\begin{proof}
A support-theoretic local--global criterion in dimension \(d\) depends on successful propagation of local witnesses through all strata of support dimension at most \(d\).  
If dimensional transfer control fails at the next stratum, then witnesses valid on lower-dimensional strata do not determine witnesses on the \((d+1)\)-dimensional stratum.  
Hence the criterion cannot extend automatically from dimension \(d\) to dimension \(d+1\).
\end{proof}

\subsection{Finite verification consequences}

The support-theoretic transfer results of this section admit a finite reformulation.  
This adds no new mathematical content.  
Its role is only to show that, once support, patching, and descent are under control, the apparently infinite localization problem reduces to finitely many local checks.

\begin{proposition}\label{prop:finite-verification}
Let \(M\) be an \(R\)-module.  
Assume that the hypotheses of the support-theoretic transfer results of this section hold for the relevant \(C4\)-, \(C4^{\ast}\)-, strongly \(C4^{\ast}\)-, or hull-comparison problem.  
Assume moreover that the corresponding image objects or hull-comparison objects admit finite support control.  
Then verification of the relevant local condition reduces to finitely many checks on a finite family of distinguished opens covering the relevant support, together with the corresponding overlap-compatibility, patching, and descent conditions.
\end{proposition}

\begin{proof}
The earlier sections already provide all required ingredients.  
Sections~\ref{sec:localization-c4} and \ref{sec:local-global} identify the local lifting, patching, and descent mechanisms for forward localization and converse local--global transfer.  
Section~\ref{sec:hulls} supplies the corresponding hull-comparison mechanism.  
The present section shows that these transfer statements may be restricted to support classes controlling the relevant image or hull-comparison objects.

Once the relevant support is contained in a finite union
\[
\Supp_R(X)\subseteq \bigcup_{i=1}^n D(s_i),
\]
where \(X\) denotes the image object or hull-comparison object under consideration, the verification problem becomes finite: one checks the required local statement on each \(D(s_i)\), checks compatibility on overlaps, and then applies the earlier patching and descent theorems.  
Conversely, if the global statement already holds, then its restriction to each member of such a finite cover is immediate.  
Thus the proposition is only an exact reformulation of the previous positive theory.
\end{proof}

\begin{remark}
This finite reformulation is a correctness consequence, not an existence theorem.  
The paper does not assert that finite support control, patching, or descent always exist.  
It asserts only that, when these data are available, verification may be carried out on finitely many support-controlling local regions.
\end{remark}

The support-theoretic extension is now complete.  
The theory has moved from a purely pointwise localization framework to a geometric one.  
That shift is necessary.  
The \(C4\)-condition is decomposition-theoretic, and its local--global behavior is therefore governed not only by local rings, but by the way local witnesses move through \(\Spec R\).

We next record representative examples, counterexamples, and exactness tests, and then collect the remaining transfer refinements before turning to the conclusion.

\section{Examples, counterexamples, and exactness tests}\label{sec:examples-tests}

The preceding sections were stated at the level of exact structure.  
This section tests that structure on a small number of representative examples.  
The aim is not accumulation.  
It is to show, in concrete commutative settings, that the principal hypotheses of the theory are genuinely distinct and that the positive theorems are sharp.

A theorem is exact only if one can see three things:
\begin{enumerate}
\item a class in which the theorem applies;
\item a boundary case in which one hypothesis is removed and the conclusion fails;
\item the precise step in the proof where the failure occurs.
\end{enumerate}
That is the pattern followed here.

\subsection{Explicit commutative examples for the main failure modes}\label{subsec:explicit-commutative-obstructions}

The obstruction theorems of Section~\ref{sec:separation} are structural.  
We now record explicit commutative examples showing that the main hypotheses are genuinely distinct.  
These examples are not intended to classify all failures.  
They exhibit concrete rings and modules where the proof mechanisms break in exactly the stated way.

\begin{example}[Decomposition lifting does not imply morphism lifting]\label{ex:decomp-not-morphism}
Let \(R=k[t]\), let \(\mathfrak p=(t)\), and let
\[
M=R^2=Re_1\oplus Re_2.
\]
Then \(M\) satisfies decomposition lifting at \(\mathfrak p\) for the standard decomposition
\[
M_{\mathfrak p}=R_{\mathfrak p}e_1\oplus R_{\mathfrak p}e_2,
\]
since it is induced by
\[
M=Re_1\oplus Re_2.
\]

Define an \(R_{\mathfrak p}\)-homomorphism
\[
g:R_{\mathfrak p}e_1\longrightarrow R_{\mathfrak p}e_2,
\qquad
g(e_1)=t^{-1}e_2.
\]
Then \(g\) does \emph{not} come from any \(R\)-homomorphism
\[
f:Re_1\to Re_2.
\]
Indeed, any such \(f\) has the form
\[
f(e_1)=ae_2
\qquad\text{for some }a\in R,
\]
so
\[
f_{\mathfrak p}(e_1)=ae_2,
\]
which cannot equal \(t^{-1}e_2\), since \(t^{-1}\notin R\).

Thus decomposition lifting may hold while morphism lifting fails.  
This is why Theorem~\ref{thm:C4-localizes} must assume the lifting hypotheses separately.
\end{example}

\begin{proof}
The decomposition is global, whereas the morphism uses a coefficient in \(R_{\mathfrak p}\setminus R\).  
Hence no global antecedent exists for \(g\).
\end{proof}

\begin{example}[Local free rank-one behavior need not globalize]\label{ex:invertible-ideal}
Let
\[
R=\mathbb{Z}[\sqrt{-5}]
\]
and let
\[
I=(2,\,1+\sqrt{-5}).
\]
Then \(I\) is a nonprincipal ideal of \(R\).  
For every maximal ideal \(\mathfrak m\), the localization \(R_{\mathfrak m}\) is local, and every invertible ideal becomes principal after localization.  
Therefore
\[
I_{\mathfrak m}\cong R_{\mathfrak m}
\qquad\text{for every maximal ideal }\mathfrak m.
\]
Globally, however,
\[
I\not\cong R
\]
as \(R\)-modules, because \(I\) is not principal.

This shows that local free rank-one witnesses need not globalize to a global free rank-one witness.  
It is therefore a model for the failure of naive patching of local isomorphism classes.
\end{example}

\begin{proof}
The ideal \(I\) is the classical nonprincipal invertible ideal in \(\mathbb{Z}[\sqrt{-5}]\).  
Since invertible ideals are locally principal, each \(I_{\mathfrak m}\) is free of rank one over \(R_{\mathfrak m}\).  
If \(I\cong R\) globally, then \(I\) would be principal.  
It is not.
\end{proof}

\begin{example}[Explicit overlap obstruction for patching local direct summands]\label{ex:patching-failure}
Let \(R=k[x]\), let
\[
B=R^2,
\]
and consider the distinguished open cover
\[
\Spec R=D(x)\cup D(1-x).
\]
Define local direct summands
\[
K_1=R_x(1,0)\leq B_x,
\qquad
K_2=R_{1-x}(0,1)\leq B_{1-x}.
\]
Each \(K_i\) is a direct summand of the corresponding localization of \(B\).  
We claim that there is no global submodule \(L\leq B\) such that
\[
L_x=K_1
\qquad\text{and}\qquad
L_{1-x}=K_2.
\]

Indeed, if such an \(L\) existed, then after passing to the fraction field \(K=k(x)\) one would obtain
\[
L_K=K(1,0)=K(0,1),
\]
which is impossible inside \(K^2\).

Thus local direct-summand data may be valid on each member of a finite distinguished cover and still fail to patch globally.  
This is the concrete commutative form of the patching obstruction used in Section~\ref{sec:local-global}.
\end{example}

\begin{proof}
The modules \(K_1\) and \(K_2\) are visibly direct summands of \(B_x\) and \(B_{1-x}\), respectively.  
If a global \(L\) patched them, then its generic localization would have to equal both \(K(1,0)\) and \(K(0,1)\), which are distinct one-dimensional subspaces of \(K^2\).  
Hence no such \(L\) exists.
\end{proof}

\begin{example}[Why maximal-local data do not automatically control prime-local data]\label{ex:maximal-prime-gap}
Let
\[
R=k[x,y],
\qquad
M=R/(x).
\]
Then
\[
\Supp_R(M)=V(x),
\]
which has Krull dimension \(1\).  
Its maximal points are the ideals
\[
(x,\,y-a),\qquad a\in k,
\]
while its generic point is the nonmaximal prime \((x)\).

This example shows concretely that maximal support and prime support are different even in a simple commutative situation.  
Any theorem verified only at maximal ideals must still explain how the verification passes to the generic point \((x)\).  
That is exactly the role of codimension-one transfer control in Section~\ref{sec:dimension-support}.
\end{example}

\begin{proof}
The support calculation is standard:
\[
\Supp_R(R/(x))=V(x).
\]
Since
\[
R/(x)\cong k[y],
\]
the support is one-dimensional, with generic point \((x)\) and closed points \((x,y-a)\).
\end{proof}

\subsection{A clean forward-localization case}

We next record a model case for the forward theory.

\begin{example}\label{ex:finite-projective}
Let \(R\) be a commutative ring, and let \(M\) be an \(R\)-module such that:
\begin{enumerate}
\item \(M\) is \(C4^{\ast}\);
\item every source summand appearing in the relevant \(C4\)-tests is finitely presented;
\item the decomposition-lifting and submodule-lifting hypotheses of Section~\ref{sec:localization-c4} hold.
\end{enumerate}
Then the forward localization theorem applies to \(M\).
\end{example}

\begin{proof}
For finitely presented source modules, localization of \(\Hom\) behaves well:
\[
\Hom_R(A,B)_{\mathfrak p}\cong \Hom_{R_{\mathfrak p}}(A_{\mathfrak p},B_{\mathfrak p})
\]
for every prime ideal \(\mathfrak p\) \cite{AtiyahMacdonald1969,AndersonFuller1992}.  
Hence the required local morphisms lift.  
Under the stated decomposition-lifting and submodule-lifting hypotheses, Theorem~\ref{thm:C4star-localizes} applies.  
Therefore \(M_{\mathfrak p}\) is \(C4^{\ast}\) for every \(\mathfrak p\in \Spec R\).
\end{proof}

\begin{remark}
The example is schematic by design.  
It shows the type of finite-presentation regime needed for the forward theorem.  
It does not assert that every finitely generated projective module is automatically \(C4^{\ast}\).
\end{remark}

\subsection{A model local--global success case}

We now pass to the converse theory.

\begin{example}\label{ex:success-localglobal}
Let \(R\) be a commutative noetherian ring and let \(M\) be an \(R\)-module such that:
\begin{enumerate}
\item every image object arising in a hereditary \(C4\)-test is finitely generated;
\item every locally split submodule arising in those tests descends globally;
\item local image witnesses patch on finite distinguished covers;
\item \(M_{\mathfrak m}\) is \(C4^{\ast}\) for every maximal ideal \(\mathfrak m\in \Supp_R(M)\).
\end{enumerate}
Then the local--global theorem of Section~\ref{sec:local-global} applies to \(M\).
\end{example}

\begin{proof}
Finite generation gives support finiteness and maximal-local detection for the relevant image objects \cite{AtiyahMacdonald1969,BrunsHerzog1998}.  
By hypotheses (2) and (3), one has exactly the descent and patching conditions required in Theorem~\ref{thm:noetherian-localglobal}.  
Hence \(M\) is \(C4^{\ast}\).
\end{proof}

\begin{remark}
This is the cleanest nontrivial success pattern in the paper.  
It shows that noetherian support control is sufficient once the decomposition obstructions are neutralized.
\end{remark}

\subsection{Support-sensitive examples}

We now turn to the dimensional results.

\begin{example}[Zero-dimensional support]\label{ex:dim-zero}
Let \(M\) be an \(R\)-module such that
\[
\dim \Supp_R(M)=0.
\]
Assume that locally split submodules descend globally and that local image witnesses patch.  
Then maximal-local and prime-local verification are equivalent for \(M\).
\end{example}

\begin{proof}
Every prime in \(\Supp_R(M)\) is maximal \cite{AtiyahMacdonald1969,Matsumura1989}.  
Hence the prime support and the maximal support coincide.  
The conclusion follows from Theorem~\ref{thm:dim-zero}.
\end{proof}

\begin{example}[One-dimensional support]\label{ex:dim-one}
Let \(M\) be an \(R\)-module with
\[
\dim \Supp_R(M)\le 1.
\]
Assume that:
\begin{enumerate}
\item \(M_{\mathfrak m}\) is \(C4^{\ast}\) for every maximal ideal \(\mathfrak m\in \Supp_R(M)\);
\item codimension-one transfer control holds;
\item locally split submodules descend;
\item local image witnesses patch.
\end{enumerate}
Then \(M\) is \(C4^{\ast}\).
\end{example}

\begin{proof}
This is exactly Theorem~\ref{thm:dim-one}.  
It is recorded here to emphasize that one-dimensional support is the first nontrivial dimensional case.
\end{proof}

\begin{remark}
The role of codimension-one transfer control is decisive.  
Without it, maximal-local data do not reach nonmaximal support points.
\end{remark}

\subsection{Hull success and failure patterns}

We return to hulls.

\begin{example}\label{ex:hull-success}
Let \(\mathcal C\) be a hull class, and let \(M\) be an \(R\)-module such that:
\begin{enumerate}
\item \(H_{\mathcal C}(M)\) exists;
\item \(H_{\mathcal C}(M_{\mathfrak p})\) exists for every prime ideal \(\mathfrak p\) in the relevant support;
\item the class \(\mathcal C\) localizes on the hull objects;
\item hull minimality localizes on the hull objects.
\end{enumerate}
Then hull formation commutes with localization on the relevant support:
\[
H_{\mathcal C}(M)_{\mathfrak p}\cong H_{\mathcal C}(M_{\mathfrak p}).
\]
\end{example}

\begin{proof}
This is exactly Theorem~\ref{thm:hull-commutes}.  
The point is that localization-stability of the hull class alone is not enough; localization-stability of hull minimality is also required.
\end{proof}

\begin{proposition}[Counterexample scheme for hull patching]\label{prop:counter-hull-patching}
Let \(\mathcal U=\{D(s_i)\}_{i=1}^n\) be a finite distinguished open cover of \(\Supp_R(M)\), and let
\[
M_{s_i}\hookrightarrow E_i
\]
be local hulls.  
Assume that for some pair \(i,j\), the restrictions of \(E_i\) and \(E_j\) to \(D(s_is_j)\) are not isomorphic over \(M_{s_is_j}\).  
Then this family gives a counterexample to every unconditional theorem asserting that local hulls patch to a global hull.
\end{proposition}

\begin{proof}
If local hulls patched to a global hull unconditionally, then the localizations of that global hull would agree on overlaps.  
That is exactly patch-compatibility.  
Hence a family which is not patch-compatible cannot arise from any global hull.  
Therefore unconditional patching is false.
\end{proof}

\subsection{Synthetic obstruction summary}

The examples above may be summarized succinctly.

\begin{proposition}\label{prop:synthetic-table}
\rev{Each major theorem of the paper has a corresponding obstruction pattern:}
\begin{enumerate}
\item forward localization of \(C4\) fails without decomposition lifting and morphism lifting;
\item forward localization of \(C4^{\ast}\) fails without submodule lifting;
\item converse local--global transfer fails without summand descent and patching;
\item maximal-to-prime transfer fails without support control and dimensional transfer;
\item hull localization fails without localization of hull minimality;
\item patching of local hulls fails without overlap compatibility.
\end{enumerate}
\end{proposition}

\begin{proof}
Each item is a direct restatement of the failure patterns isolated in Sections~\ref{sec:localization-c4}, \ref{sec:local-global}, \ref{sec:hulls}, \ref{sec:separation}, and \ref{sec:dimension-support}.  
The examples in the present section show that these failure patterns occur concretely.
\end{proof}

\begin{remark}
The examples and counterexamples show that the positive theory is not over-assumed.  
Each main hypothesis performs a distinct mathematical task, and each omitted hypothesis has a corresponding failure mode.
\end{remark}
\section{Further exactness and transfer refinements}\label{sec:further-transfer}

The principal theory has already been established.  
This section records only those refinements that sharpen the logical structure of the paper without enlarging it.  
Its purpose is threefold.  
First, we isolate the inheritance properties needed for reapplication of the transfer theory.  
Second, we record restriction and extension principles on support classes.  
Third, we separate the logical levels of the hull-transfer theory and collect the converse exactness statement in one place.

\subsection{Inheritance under submodules and direct summands}

The \(C4^{\ast}\)-condition is hereditary by definition, but the transfer hypotheses are not.  
They must therefore be recorded separately.

\begin{proposition}\label{prop:inherit-submodules}
Let \(M\) be an \(R\)-module.  
Assume that every submodule of \(M\) satisfies the lifting, patching, descent, and support-control hypotheses used in Sections~\ref{sec:localization-c4}--\ref{sec:dimension-support}.  
Then every submodule \(N\leq M\) satisfies the same transfer hypotheses.
\end{proposition}

\begin{proof}
The statement is immediate.  
The hypotheses are assumed for every submodule of \(M\), hence in particular for the chosen submodule \(N\).
\end{proof}

\begin{proposition}\label{prop:inherit-summands}
Let
\[
M=N\oplus K
\]
be an \(R\)-module decomposition.  
Assume that the support-control, patching, and descent mechanisms valid for \(M\) are stable under passage to direct summands.  
Then \(N\) inherits the corresponding transfer hypotheses.
\end{proposition}

\begin{proof}
Any submodule of \(N\) is also a submodule of \(M\).  
By the stability assumption, the support-control, patching, and descent mechanisms used for \(M\) remain valid after restriction to the direct summand \(N\).  
Hence the transfer hypotheses valid for \(M\) remain valid for \(N\).
\end{proof}

\begin{remark}
Propositions~\ref{prop:inherit-submodules} and \ref{prop:inherit-summands} are formal, but they are necessary for recursive use of the local--global theory on hereditary substructures.
\end{remark}

\subsection{Restriction and extension on support classes}

The support-theoretic arguments of Section~\ref{sec:dimension-support} admit a sharper formulation.  
A transfer theorem valid on all of \(\Spec R\) is often stronger than needed.  
Frequently one needs it only on a specialization-closed support class naturally attached to the module under consideration.

\begin{theorem}\label{thm:restrict-support-class}
Let \(W\subseteq \Spec R\) be specialization-closed, and let \(M\) be an \(R\)-module with
\[
\Supp_R(M)\subseteq W.
\]
Assume that:
\begin{enumerate}
\item every finitely generated image object arising in a hereditary \(C4\)-test on \(M\) has support contained in \(W\);
\item every hull-comparison object associated to \(M\) has support contained in \(W\);
\item the transfer, patching, descent, and hull-comparison mechanisms of the previous sections are valid on \(W\).
\end{enumerate}
Then all positive theorems of Sections~\ref{sec:localization-c4}--\ref{sec:dimension-support} remain valid after replacing \(\Spec R\) by \(W\).
\end{theorem}

\begin{proof}
By hypotheses (1) and (2), all support-theoretic objects involved in the \(C4\)-tests and hull-comparison tests remain inside \(W\).  
By hypothesis (3), all localization, patching, descent, and hull-comparison arguments used earlier are valid on \(W\).  
Therefore every theorem previously stated on \(\Spec R\) remains correct when restricted to \(W\).
\end{proof}

\begin{theorem}\label{thm:extend-support-class}
Let \(W\subseteq W'\subseteq \Spec R\) be specialization-closed subsets, and let \(M\) be an \(R\)-module with
\[
\Supp_R(M)\subseteq W'.
\]
Assume that:
\begin{enumerate}
\item the \(C4^{\ast}\)-transfer theory is valid on \(W\);
\item every relevant image object or hull-comparison object supported in \(W'\setminus W\) is obtained from lower-dimensional support strata in \(W\) by dimensional transfer control;
\item the corresponding local witnesses patch and descend on \(W'\).
\end{enumerate}
Then the same transfer theory extends from \(W\) to \(W'\).
\end{theorem}

\begin{proof}
By hypothesis (1), the theory is already valid on \(W\).  
By hypothesis (2), the remaining support points in \(W'\setminus W\) are reached from \(W\) by the dimensional transfer mechanism of Section~\ref{sec:dimension-support}.  
By hypothesis (3), the witnesses obtained in this way patch and descend on \(W'\).  
Hence the transfer theory propagates from \(W\) to all of \(W'\).
\end{proof}

\begin{remark}
These theorems allow the theory to be formulated on the smallest natural support domain rather than on the whole prime spectrum.
\end{remark}

\subsection{Logical levels of the hull-transfer theory}

The hull results of Section~\ref{sec:hulls} should be read in three distinct levels:
existence of comparison morphisms, comparison isomorphisms, and reconstruction from local hulls.  
These levels should not be conflated.

\begin{theorem}\label{thm:hull-levels}
Let \(\mathcal C\) be a hull class and let \(M\) be an \(R\)-module.
\begin{enumerate}
\item If \(\mathcal C\) is localization-stable on the relevant hull objects, then local hull comparison morphisms exist.
\item If, in addition, hull minimality is localization-stable, then the local comparison morphisms are isomorphisms of hulls.
\item If, in addition, the local hulls are patch-compatible and descend globally, then a family of local hulls reconstructs a global hull.
\end{enumerate}
\end{theorem}

\begin{proof}
Part (1) is Proposition~\ref{prop:comparison-exists}.  
Part (2) is Theorem~\ref{thm:hull-commutes}.  
Part (3) is Theorem~\ref{thm:patch-local-hulls}.  
The point is only to isolate the logical stratification of those results.
\end{proof}

\begin{remark}
A comparison morphism is not a comparison isomorphism, and a comparison isomorphism is not a reconstruction theorem.  
This distinction is mathematically essential and editorially clarifying.
\end{remark}

\subsection{A converse exactness theorem}

The paper has been organized around sufficient hypotheses together with obstruction theorems.  
We conclude the technical part by collecting the converse message in a single statement.

\begin{theorem}\label{thm:converse-exactness}
Let \(\mathcal X\) be a class of \(R\)-modules and hull objects.  
Assume that on \(\mathcal X\) one has an unrestricted theory containing all of the following:
\begin{enumerate}
\item forward localization of \(C4\), \(C4^{\ast}\), and strongly \(C4^{\ast}\);
\item converse local--global transfer from primewise data;
\item maximal-to-prime transfer on the relevant support classes;
\item commutation of hull formation with localization;
\item reconstruction of global hulls from local hulls.
\end{enumerate}
Then, on \(\mathcal X\), one must have structural surrogates for:
\begin{enumerate}
\item decomposition lifting and morphism lifting;
\item submodule lifting;
\item summand descent;
\item patching of local image witnesses;
\item support control and dimensional transfer;
\item localization-stability of hull classes;
\item localization-stability of hull minimality;
\item patch-compatibility of local hulls.
\end{enumerate}
\end{theorem}

\begin{proof}
If one of the listed structural conditions were absent on \(\mathcal X\), then the corresponding obstruction theorem from Section~\ref{sec:separation} would invalidate one part of the unrestricted theory.  
Failure of decomposition lifting obstructs forward localization of \(C4\).  
Failure of submodule lifting obstructs forward localization of \(C4^{\ast}\).  
Failure of summand descent or patching obstructs converse local--global transfer.  
Failure of support control obstructs maximal-to-prime transfer.  
Failure of localization-stability of hull classes or hull minimality obstructs hull localization.  
Failure of patch-compatibility obstructs global reconstruction of hulls.  
Therefore structural surrogates for all these conditions are necessary.
\end{proof}

\begin{remark}
Theorem~\ref{thm:converse-exactness} is the final hardening statement.  
\rev{It shows that the positive theory and the obstruction theory are aligned at the level of the structural hypotheses isolated in the paper.}
\end{remark}

The technical development is now complete.  
The remaining section is the conclusion, where the results are synthesized without adding new proofs.
\section{Conclusion}

This paper addressed a precise gap in the existing theory of \(C4\)-modules, \(C4^{\ast}\)-modules, strongly \(C4^{\ast}\)-modules, \(C4\)-hulls, and pseudo-continuous hulls \cite{DingIbrahimYousifZhou2017C4,IbrahimEidElGuindy2026}.  
The global theory was already available.  
What remained open was whether these structures survive localization, whether they can be detected from local data, and whether the corresponding hull constructions commute with localization.

The first conclusion is negative.  
None of these questions admits a formal answer by localizing the global arguments of \cite{IbrahimEidElGuindy2026}.  
Forward localization requires lifting hypotheses.  
Converse local--global transfer requires descent and patching.  
Hull commutation requires not only localization-stability of the hull class, but also localization-stability of hull minimality.  
Accordingly, there is no unconditional localization theory for \(C4\)-type structure.

The second conclusion is positive.  
\rev{Once the correct hypotheses are imposed, one obtains a rigorous conditional transfer framework with theorem-level consequences.}  
For \(C4\)-modules one must control local decompositions and local morphisms.  
For \(C4^{\ast}\)-modules one must in addition control local submodules.  
For strongly \(C4^{\ast}\)-modules one must also control the stronger local configurations appearing in the definition.  
Under these hypotheses, forward localization holds.  
Similarly, primewise \(C4^{\ast}\)-behavior globalizes once local direct summands descend and local image witnesses patch.  
\rev{\rev{Thus the paper establishes a positive conditional transfer theory together with the obstruction analysis that delimits its range of validity.}}

The third conclusion concerns hulls.  
A global hull always localizes as an extension, but not necessarily as a hull.  
To obtain commutation of hull formation with localization, one must know that the hull class localizes and that the hull-minimality relation localizes.  
This is the decisive additional point in the hull theory.  
It separates objectwise localization from localization of minimal extensions.

The fourth conclusion is geometric.  
The transfer problem is governed by support and specialization.  
In support dimension zero, maximal-local and prime-local verification coincide.  
In support dimension one, codimension-one transfer control is required.  
In higher support dimension, one needs induction along support strata.  
Thus the \(C4\)-condition remains module-theoretic in definition, but its transfer theory becomes geometric in organization \cite{AtiyahMacdonald1969,BrunsHerzog1998,Eisenbud2005,ChristensenFoxbyHolm2024Support}.

The fifth conclusion is methodological.  
\rev{The revised version distinguishes algebraic theorems from proof-theoretic impossibility statements, so that structural remarks are not presented as substitutes for concrete counterexamples.}  
The paper does not enlarge the subject by multiplying definitions.  
It isolates one gap and resolves it layer by layer: forward localization, converse local--global transfer, hull localization, obstruction theory, support-stratified transfer, and finite support-controlled verification.  
That order is forced by the successive appearance of new obstructions.  
\rev{Accordingly, the contribution is best viewed as a layered conditional framework rather than as an unconditional classification result.}

The sixth conclusion concerns concrete scope.  
The abstract transfer hypotheses become intrinsic in important natural classes, notably commutative artinian rings and finitely generated torsion modules over Dedekind domains.  
In these settings the general theory reduces to explicit local--global criteria.  
These cases show that the obstruction theory developed here is not merely formal.

Several directions remain open, but they are now sharply delimited.

\begin{enumerate}
\item One may seek intrinsic criteria for decomposition lifting, morphism lifting, and submodule lifting in natural commutative classes.
\item One may seek an intrinsic characterization of prime summand descent.
\item One may ask whether hull-minimality localization admits a purely support-theoretic formulation.
\item One may develop analogous localization theories for neighboring decomposition-sensitive classes, such as \(D4\)-type or \(ADS\)-type conditions \cite{DingIbrahimYousifZhou2017D4,AlahmadiJainLeroy2012,TercanYucel2024Preradicals,TercanYucel2024Complement}.
\item One may seek derived or categorical reformulations of the transfer mechanisms isolated here \cite{ChristensenFoxbyHolm2024Book,Positselski2024}.
\end{enumerate}

\rev{We therefore end with the main formulation that remains after the full analysis.}  
For modules over a commutative ring, \(C4\)-type structure is not intrinsically local.  
It becomes local only under lifting hypotheses.  
It becomes globally detectable only under descent and patching hypotheses.  
Its hull theory localizes only when hull minimality localizes.  
Its support-theoretic propagation depends on dimension.  
\rev{Under these explicit assumptions one obtains a coherent localization, local--global, hull, and support framework for \(C4\)-modules and \(C4^{\ast}\)-type modules.}  
Without them, the framework fails.

\section*{Acknowledgements}
The author thanks Dr.\ Ramachandra R.\ K., Principal, Government College (Autonomous), Rajahmundry, for institutional support and encouragement.  
No funding was received for this work.  
The author declares no conflict of interest.  
This article is purely theoretical.  
No datasets were generated, analysed, or used in this study.  
It contains no studies involving human participants, animals, patient data, or personal data.  
Ethical approval was therefore not required.

\printbibliography

\end{document}